\documentclass{amsart}
\usepackage{graphicx}
\usepackage{mathrsfs}
\usepackage{amsthm}
\usepackage{amsmath}
\usepackage{amssymb}
\usepackage{color}
\usepackage{xcolor}
\usepackage{subfigure}
\usepackage[bookmarks=false,pdfstartview={FitH}]{hyperref}

    \newtheoremstyle{upright}%
        {1pt plus1pt minus1pt}%
        {1pt plus1pt minus1pt}%
        {\upshape}%
        {}%
        {\bfseries\scshape}%
        {\textbf{.}}%
        {1em}%
        {}%

\newtheorem{theorem}{Theorem}[section]
\newtheorem*{TysonsResult}{\textbf{Theorem 4.1 of} [\textbf{Du-CaTy}]}

\newtheorem{lemma}[theorem]{Lemma}
\newtheorem{question}[theorem]{Question}

\newtheorem{proposition}[theorem]{Proposition}

\theoremstyle{definition}
\newtheorem{definition}[theorem]{Definition}
\newtheorem{example}[theorem]{Example}
\newtheorem{notation}[theorem]{Notation}
\newtheorem{remark}[theorem]{Remark}

\renewenvironment{proof}[1][Proof]{\textbf{#1.} }{\ \rule{0.5em}{0.5em}}

\newcommand{\N}{\mathbb{N}} 
\newcommand{\Z}{\mathbb{Z}} 

\newcommand{\ofraci}[1]{\mathscr{O}_{#1}(\xio{#1},\theta_{#1}^0)}

\newcommand{\ofracixang}[3]{\mathscr{O}_{#1}(#2,#3)}
\newcommand{\ofraciang}[2]{\mathscr{O}_{#1}(\xio{#1},#2)}

\newcommand{\xoo}{x_0^0}

\newcommand{\xio}[1]{x_{#1}^0}
\newcommand{\yio}[1]{y_{#1}^0}

\newcommand{\compseqsi}[1]{\{\mathscr{O}_n(\xio{n},\alpha^0)\}_{n=#1}^\infty}
\newcommand{\compseqsiang}[2]{\{\mathscr{O}_n(\xio{n},#2)\}_{n=#1}^\infty}
\newcommand{\compseqsixang}[3]{\{\mathscr{O}_n(#2,#3)\}_{n=#1}^\infty}

\newcommand{\mbfa}{\mathbf{a}}

\newcommand{\mbf}[1]{\mathbf{#1}}

\newcommand{\celli}[2]{C_{#1,#2}} 

\newcommand{\alphaio}[1]{\alpha_{#1}^0}
\newcommand{\ssai}[1]{\mathcal{S}(S_{\mbfa,#1})}

\newcommand{\sai}[1]{S_{\mbfa,#1}}
\newcommand{\omegasa}[1]{\Omega(S_{#1})}
\newcommand{\omegasai}[2]{\Omega(S_{#1,#2})}
\newcommand{\omegasi}[1]{\Omega(S_{\mathbf{a},#1})}

\newcommand{\osi}[1]{\mathscr{O}_{#1}(\xio{#1},\alpha_{#1}^0)}

\newcommand{\osixang}[3]{\mathscr{O}_{#1}(#2,#3)}
\newcommand{\slopesa}[1]{\text{Slope}(S_{#1})}

\title{Periodic billiard orbits of self-similar Sierpinski carpets}

\author{Joe P. Chen}

\address{Departments of Physics \& Mathematics, Cornell University, Ithaca, NY 14853, USA}
\email{joe.p.chen@cornell.edu}

\author{Robert G. Niemeyer}

\address{Department of Mathematics \& Statistics, University of New Mexico, Albuquerque, Albuquerque, NM  87131, USA}
\email{niemeyer@math.unm.edu}

\thanks{The work of R. G. Niemeyer was partially supported by the National Science Foundation under the NSF MCTP grant DMS-1148801, while an MCTP postdoctoral fellow at the University of New Mexico, Albuquerque.}

\numberwithin{equation}{section}

\begin{document}

\begin{abstract}
We identify a collection of periodic billiard orbits in a self-similar Sierpinski carpet billiard table $\Omega(S_\mbfa)$. Based on a refinement of the result of Durand-Cartagena and Tyson regarding nontrivial line segments in $S_\mbfa$, we construct what is called an eventually constant sequence of compatible periodic orbits of prefractal Sierpinski carpet billiard tables $\Omega(S_{\mbfa,n})$.  The trivial limit of this sequence then constitutes a periodic orbit of $\Omega(S_\mbfa)$.  We also determine the corresponding translation surface $\mathcal{S}(S_{\mbfa,n})$ for each prefractal table $\Omega(S_{\mbfa,n})$, and show that the genera $\{g_n\}_{n=0}^\infty$ of a sequence of translation surfaces $\{\mathcal{S}(S_{\mbfa,n})\}_{n=0}^\infty$ increase without bound. Various open questions and possible directions for future research are offered.
\end{abstract}

\maketitle

\tableofcontents
\section{Introduction}
\label{sec:Introduction}

The subject of polygonal billiards is well-developed; see, e.g., \cite{Gut1,Gut2,HuSc,KaZe,MasTa,Sm,Ta1,Ta2,Zo} and the pertinent references therein for excellent surveys on the subject and the current state of the art.  There is a particular family of fractal sets, each of which can be viewed as the limit of a sequence of polygonal approximations.  M. L. Lapidus and the second author have begun investigating the billiard dynamics on the Koch snowflake fractal billiard table and the so-called T-fractal billiard table; see \cite{LapNie1,LapNie2,LapNie3}. Besides determining periodic orbits of these two fractal billiard tables, these articles proposed a possible framework in which one could begin investigating the billiard dynamics on an arbitrary billiard table with fractal boundary.

In this article, we proceed to identify a collection of periodic billiard orbits in the self-similar Sierpinski carpet billiard table $\Omega(S_{\bf a})$, where ${\bf a}:=\{a_{i}\}_{i=1}^\infty$ is a constant sequence of positive odd integers with $a_i\geq 3$ for every $i\geq 1$; see Figure \ref{fig:S5-123}. To fix notations, we regard $S_{\bf a}$ as a subset of the unit square $Q$ in $\mathbb{R}^2$, and fix the lower-left corner of $Q$ to be the origin $(0,0)$. The boundary of an open square removed in the construction of $S_{\bf a}$ is referred to as a \emph{peripheral square} of $S_{\bf a}$.

\begin{figure}
\begin{center}
\includegraphics[scale=.2]{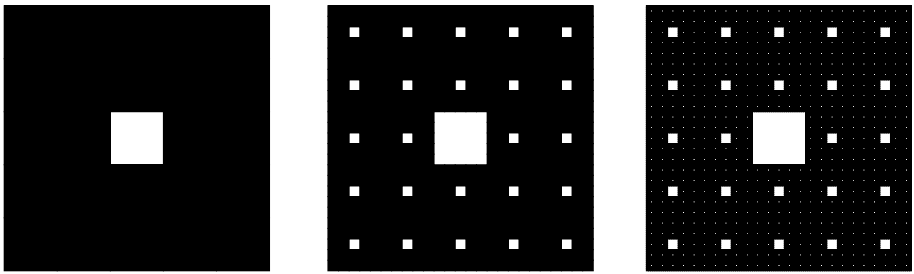}
\end{center}
\caption{The first three approximations of the self-similar Sierpinski carpet $S_\mbf{5}$.  The explicit construction of a Sierpinski carpet is given in \S\ref{subsec:ASelfSimilarSierpinskiCarpet}.}
\label{fig:S5-123}
\end{figure}

The basis for our work is the recent result of \cite{Du-CaTy} by E. Durand-Cartagena and J. Tyson, who identified all the possible slopes of nontrivial line segments in a given Sierpinski carpet $S_{\bf a}$. This set of slopes, denoted by ${\rm Slopes}(S_\mbfa)$, is a finite set of rational values, and equals the disjoint union of $A_\mbfa$ and $B_\mbfa$, where (see Theorem 4.1 of \cite{Du-CaTy})
\begin{align*}
A_{\bf a} &= \left\{\pm\frac{p}{q}, \pm\frac{q}{p} \,\, :\,\, p+q\leq a,\, 0\leq p< q\leq a-1,\, p,q\in \N\cup \{0\},\, p+q \text{ is odd}\right\}, \\
B_{\bf a} &= \left\{\pm\frac{p}{q}, \pm\frac{q}{p} \,\,:\,\, p+q\leq a-1,\, 0\leq p\leq q\leq a-2,\, p,q\in \N,\, p \text{ and } q \text{ are odd}\right\}.
\end{align*}
\noindent We note that $A_{\bf{a-2}} \subsetneq A_{\bf a}$ and $B_{\bf{a-2}} \subsetneq B_{\bf a}$. Roughly speaking, a nontrivial line segment of slope $\alpha \in A_{\bf a}$ can emanate from the origin $(0,0)$, and a nontrivial line segment of slope $\alpha \in B_{\bf a}$ can emanate from $(\frac{1}{2},0)$, although the association between the slope set and starting point is not absolute. For example in $S_{\bf 7}$, the line segments emanating from, respectively, $(0,0)$ and $(\frac{1}{2},0)$, with slope $\frac{2}{3} \in A_{\bf 7}$ are both nontrivial line segments; see \S4.

Key to our description of particular family of periodic orbits of a self-similar Sierpinski carpet is analyzing whether nontrivial line segments hit or avoid peripheral squares of $S_{\bf a}$. It has already been shown in \cite{Du-CaTy} that nontrivial line segments of slope $\alpha \in B_{\bf a}$ starting from $(\frac{1}{2},0)$ avoid all peripheral squares of $S_{\bf a}$. We show below in Theorem \ref{thm:refinement} an important distinction between nontrivial line segments of slope $\alpha \in A_{\bf{a-2}}$ and those of slope $\alpha \in A_{\bf a} \backslash A_{\bf{a-2}}$. Our result is as follows. Suppose that the line segment starts from $(0,0)$. If the segment has slope $\alpha \in A_{\bf{a-2}}$, then it avoids all peripheral squares of $S_{\bf a}$. If the segment has slope $\alpha \in A_{\bf a}\backslash A_{\bf{a-2}}$, then it must intersect with the corner of some peripheral square of $S_{\bf a}$. This distinction was not stated in Theorem 4.1 of \cite{Du-CaTy}, and turns out to play an important role in identifying particular periodic billiard orbits of a self-similar Sierpinski carpet billiard.

We then apply the results about nontrivial line segments to obtain information about the billiard orbits. This is achieved via a sequence of arguments regarding the scaling of cells and reflected-unfolding of particular orbits. We view a Sierpinski carpet as the unique fixed point attractor of a sequence of rational polygonal approximations (or prefractal approximations).  As such, we investigate the limiting behavior of particular sequences of compatible periodic orbits of prefractal billiard tables.\footnote{We present the necessary concepts and definitions from the subject of fractal billiards in \S\ref{subsec:fractalBilliards}.} Our main result, stated more precisely in Theorem \ref{thm:sequenceOfCompatiblePeriodicOrbits}, says that the following orbits in $\Omega(S_\mbfa)$ are periodic:

\begin{itemize}
\item Those starting from $(0,0)$ with slope $\alpha \in A_{\bf{a-2}}$.
\item Those starting from $(\frac{r}{2a^n},0)$ with slope $\alpha \in B_{\bf a}$, for any $n \in \mathbb{N}$ and any positive odd integer $r<2a^n$.
\end{itemize}

\emph{A priori}, this list does not exhaust all of the initial conditions giving rise to periodic orbits of $\Omega(S_\mbfa)$. In addition to describing particular periodic orbits, we can understand what constitutes the larger collection of closed orbits of $\Omega(S_\mbfa)$. In Proposition \ref{prop:compatibleclosedorbits}, we show that an orbit starting from $(\frac{k}{a^n},0)$ with slope $\alpha \in A_{\bf{a-2}}$, some $n\in \mathbb{N}$, and some positive integer $k<a^n$ is always closed, but not always periodic.

We believe that investigating the self-similar Sierpinski carpet billiard may inspire further studies on related polygonal billiards of infinite genus. Indeed, while we can identify the translation surface of the pre-Sierpinski carpet billiard table $S_{{\bf a},n}$, and show that their genera grow unboundedly as $n\to\infty$ (\S 4.1), it is unclear how one may define a translation surface for the fractal billiard table $\Omega(S_{\bf a})$, since $S_{\bf a}$ has zero Lebesgue area. In \S 5, we suggest two related billiard tables which share similar features with $S_{\bf a}$, but have nonzero Lebesgue area.

The paper is organized as follows: In \S\ref{sec:background}, we provide some background on the Sierpinski carpet, mathematical billiards, and translation surfaces. Our result on nontrivial line segments in $S_{\bf a}$, which clarifies part of Theorem 4.1 in \cite{Du-CaTy}, is described in \S\ref{sec:ARefinementOfTheorem4.1OfDuCaTy}. The core of this article is contained in \S\ref{sec:EventuallyConstantSequencesOfCompatiblePeriodicOrbitsOfSa}, where we identify the two types of periodic billiard orbits in $S_{\bf a}$, and describe the translation surface corresponding to the pre-Sierpinski carpet billiard table. Open problems are addressed in \S\ref{sec:ConcludingRemarks}.

\section{Background}
\label{sec:background}

In this section, we provide the necessary background for understanding the remainder of the article.  We will give the necessary definitions from the field of  fractal geometry and mathematical billiards, as well as the pertinent terms from \cite{Du-CaTy}.  In addition to this, we will draw upon the vocabulary of the emerging field of fractal billiards when discussing a self-similar Sierpinski carpet billiard $\omegasa{\mbfa}$ in \S\ref{sec:EventuallyConstantSequencesOfCompatiblePeriodicOrbitsOfSa}; see \cite{LapNie1,LapNie2,LapNie3}.

\subsection{Self-similarity and Sierpinski carpets}
\label{subsec:ASelfSimilarSierpinskiCarpet}

\begin{definition}[A self-similar Sierpinski carpet via an IFS]
\label{def:ASelfSimSierpinskiCarpetViaAnIFS}
Let $a\geq 3$ be an positive odd integer and $\phi_i:\mathbb{R}^2 \to \mathbb{R}^2$ a similarity contraction defined as:

\begin{align}
\phi_i(x) &:= \frac{1}{a} (x+(u_i,v_i)),
\end{align}

\noindent where $u_i,v_i\in \mathbb{N}\cup\{0\}$, $0\leq u_i\leq a-1$, $0\leq v_i\leq a-1$ and $(u_i,v_i)\neq (\frac{a-1}{2},\frac{a-1}{2})$.  Then, $\{\phi_i\}_{i=1}^{a^2-1}$ is an iterated function system.  The unique fixed point attractor of the map $\Phi(\cdot) := \bigcup_{i=1}^{a^2-1} \phi_i(\cdot)$ is called a \textit{self-similar Sierpinski carpet}.
\end{definition}

We recall the notation used in constructing Sierpinski carpets in \cite{Du-CaTy}.  Let

\begin{equation*}
\mbfa = (a_1^{-1},a_2^{-1},...) \in \left\{\frac{1}{3},\frac{1}{5},\frac{1}{7},...\right\}^\N.
\end{equation*}

Let $Q$ be the unit square. Beginning with $S_0 := Q$, partition $S_0$ into $a_1^2$ equal squares of side-length $a_1^{-1}$ and remove the middle open square, leaving $a_1^2 - 1$ many squares.  We denote the union of the remaining squares by $S_{\mbfa, 1}$ and refer to $S_{\mbfa, 1}$ as the first level approximation of the Sierpinski carpet $S_\mbfa$.  Then, partition each remaining square with side-length $a_1^{-1}$ into $a_2^{2}$ many squares with side-length $a_1^{-1}\cdot a_2^{-1}$.  From each square with side-length $a_1^{-1}$, remove the middle open square of side-length $a_1^{-1}\cdot a_2^{-1}$.  The union of the remaining squares of side-length $a_1^{-1}\cdot a_2^{-1}$ constitutes the second level approximation $S_{\mbfa,2}$ of a Sierpinski carpet $S_\mbfa$.  Continuing this process ad infinitum, we construct the Sierpinski carpet $S_\mbfa$.  More precisely,

\begin{align}
S_\mbfa &= \bigcap_{n=0}^\infty S_{\mbfa, n},
\end{align}

\noindent where $S_{\mbfa,0} = S_0$.
\begin{definition}[A self-similar Sierpinski carpet]
\label{def:ASelfSimilarSierpinskiCarpet}
If $\mathbf{a}=\{a_i^{-1}\}_{i=0}^\infty$, with $a_i = 2k_i+1$ and $k_i \in \N$, is a periodic sequence, then the Sierpinski carpet $S_\mbfa$ is called a \textit{self-similar Sierpinski carpet}.
\end{definition}


Both Definitions \ref{def:ASelfSimSierpinskiCarpetViaAnIFS} and \ref{def:ASelfSimilarSierpinskiCarpet} are equivalent.  More precisely, when $\mbfa$ is a constant sequence, then it is clear how the two definitions are equivalent.  When $\mbfa$ is a nonconstant periodic sequence, then there exists $a'$ such that $S_\mbfa$ is the unique fixed point attractor of an appropriately defined iterated function system $\{\phi_i\}_{i=1}^{{a'}^2-1}$. For the purposes of this paper, it may be advantageous to use one definition over another, depending on the context of the situation.\footnote{For a more detailed discussion of self-similar sets and iterate function systems, see \cite{Hut}.  For a more complete treatment of the subject of fractal geometry, see \cite{Fa}.}

\begin{definition}[A cell of $S_{\mathbf{a},n}$]
\label{def:ACellOfSai}
Let $a_0=2k_0+1$, $k_0\in \mathbb{N}$.  Consider a partition of the unit square $Q=S_0$ into $a_0^2$ many squares of side-length $a_0^{-1}$.  A subsquare of the partition is called a \textit{cell of $S_0$} and is denoted by $\celli{0}{a_0}$.  Let $S_\mathbf{a}$ be a Sierpinski carpet.  Consider a partition of the prefractal approximation $S_{\mathbf{a},n}$ into subsquares with side-length $(a_0\cdot a_1\cdots a_n)^{-1}$.  A subsquare of the partition of $S_{\mathbf{a},n}$ is called a \textit{cell of $S_{\mathbf{a},n}$} and is denoted by $\celli{n}{a_0a_1\cdots a_n}$ and has side-length $(a_0\cdot a_1\cdots a_n)^{-1}$.
\end{definition}

\begin{figure}
\begin{center}
\includegraphics[scale=.75]{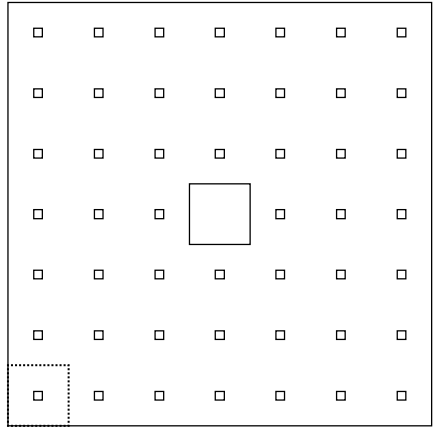}
\end{center}
\caption{We see here an example of a cell $\celli{1}{a}$ of $S_{\mbf{7},1}$.  While this cell is clearly shown as a subset of $S_{\mbf{7},2}$, it has a side-length that indicates it is called a cell of $S_{\mbf{7},1}$.}
\label{fig:anExampleOfACell}
\end{figure}

Let $\mbfa$ be a constant sequence of odd, positive integers such that $S_\mbfa$ is a self-similar Sierpinski carpet and $\{\phi_i\}_{i=1}^{a^2-1}$ an iterated function system for which $S_\mbfa$ is the unique fixed point attractor. If $\celli{n}{a^n}$ is a cell of $S_{\mbfa,n}$, then $\celli{n}{a^n}=\phi_{i_n}\circ\cdots\circ\phi_{i_1}(Q)$ for suitably chosen $i_1,i_2,...,i_n$.

We now define the important notion of a \textit{peripheral square}.

\begin{definition}[Peripheral square]
\label{def:PeripheralSquare}
In accordance with the convention established in \cite{Du-CaTy}, the boundary of an open square removed in the construction of $S_\mbfa$ is called a \textit{peripheral square} of $S_\mbfa$.  By convention, the unit square $Q=S_0$ is not a peripheral square.
\end{definition}

\begin{example}[A cell of the prefractal $S_{\mbf{7},1}$]
In Figure \ref{fig:anExampleOfACell}, we see an example of a cell of the prefractal approximation $S_{\mbf{7},1}$. We wish to emphasize the fact that the cell with the middle open square removed constitutes the scaling of $S_{\mbf{7},1}$ (relative to the origin) by $\frac{1}{7}$.  At times, we may make reference to a cell $\celli{m}{a^m}$, even when such a region is considered as a subset of $S_{\mbfa,n}$ for $n > m$.  In such a case, a cell $\celli{m}{a^m}$ will contain peripheral squares of $S_{\mbfa,n}$ with side length $a^{-k}$ for all $k$ such that $m<k\leq n$.
\end{example}

Also of great importance is the notion of a \textit{nontrivial line segment}.

\begin{definition}[Nontrivial line segment of $S_\mbfa$]
\label{def:nontrivialLineSegmentOfSa}
A \textit{nontrivial line segment of a Sierpinski carpet $S_\mbfa$} is a (straight-line) segment of the plane contained in $S_\mbfa$ that has nonzero length.
\end{definition}


\begin{notation}
Let $S_\mbfa$ be a Sierpinski carpet.  We denote by $\slopesa{\mbfa}$ the set of slopes, with values in $[0,1]$, of nontrivial line segments in $S_\mbfa$. There is no loss of generality, since applying an isometry of the square to a slope $\alpha \in \slopesa{\mbfa}$ yields a slope of $\frac{1}{\alpha}$, $-\alpha$, or $-\frac{1}{\alpha}$.
\end{notation}

\begin{TysonsResult}
\label{thm:TysonsResult}
Let $\mbfa=(\frac{1}{a},\frac{1}{a},\frac{1}{a},...)$ be a constant sequence.  Then the set of slopes of nontrivial line segments $\slopesa{\mbfa}$ is the union of the following two sets\emph{:}

\begin{align}
A_\mbfa &= \left\{\frac{p}{q} \,\, :\,\, p+q\leq a,\, 0\leq p< q\leq a-1,\, p,q\in \N\cup \{0\},\, p+q \text{ is odd}\right\}, \label{eqn:Aset}\\
B_\mbfa &= \left\{\frac{p}{q} \,\,:\,\, p+q\leq a-1,\, 0\leq p\leq q\leq a-2,\, p,q\in \N,\, p \text{ and } q \text{ are odd}\right\}.\label{eqn:Bset}
\end{align}

\noindent Moreover, if $\alpha\in A_\mbfa$, then each nontrivial line segment in $S_\mbfa$ with slope $\alpha$ touches vertices of peripheral squares, while if $\alpha\in B_\mbfa$, then each nontrivial line segment in $S_\mbfa$ with slope $\alpha$ is disjoint from all peripheral squares.  For each $\alpha\in A_\mbfa \cup B_\mbfa$, there exist maximal line segments in $S_\mbfa$ with slope $\alpha$.  Finally, if $b<a$, then any maximal nontrivial line segment in $S_\mbf{b}$ is also contained in $S_\mbfa$.  In particular, $\slopesa{\mbf{b}}\subseteq\slopesa{\mbfa}$.
\end{TysonsResult}

Observe that $A_{\mbf{a-2}} \subsetneq A_\mbfa$ and $B_{\mbf{a-2}} \subsetneq B_\mbfa$. For concreteness, we give $\slopesa{\mbfa}$ for the first four self-similar Sierpinski carpets:

\begin{eqnarray*}
A_{\mbf{3}} = \left\{\mathbf{{\color{red} 0}},\mathbf{{\color{red} \frac{1}{2}}} \right\} &,&
B_{\mbf{3}} = \left\{ \mathbf{{\color{blue} 1}} \right\},\\
A_{\mbf{5}} = \left\{{\color{red} 0}, \mathbf{{\color{red} \frac{1}{4}}}, {\color{red} \frac{1}{2}}, \mathbf{{\color{red} \frac{2}{3}}} \right\}&,&
B_{\mbf{5}} = \left\{{\mathbf{{\color{blue} \frac{1}{3}}}}, {\color{blue} 1} \right\},\\
A_{\mbf{7}} = \left\{{\color{red} 0}, \mathbf{{\color{red} \frac{1}{6}}}, {\color{red} \frac{1}{4}}, \mathbf{{\color{red} \frac{2}{5}}},  {\color{red} \frac{1}{2}}, {\color{red} \frac{2}{3}}, \mathbf{{\color{red} \frac{3}{4}}} \right\}&,&
B_{\mbf{7}} = \left\{\mathbf{{\color{blue} \frac{1}{5}}}, {\color{blue} \frac{1}{3}}, {\color{blue} 1} \right\},\\
A_{\mbf{9}} = \left\{{\color{red} 0}, \mathbf{{\color{red} \frac{1}{8}}}, {\color{red} \frac{1}{6}},  {\color{red} \frac{1}{4}}, \mathbf{{\color{red} \frac{2}{7}}}, {\color{red} \frac{2}{5}},  {\color{red} \frac{1}{2}}, {\color{red} \frac{2}{3}}, {\color{red} \frac{3}{4}}, \mathbf{{\color{red} \frac{4}{5}}} \right\}&,&
B_{\mbf{9}} = \left\{\mathbf{{\color{blue} \frac{1}{7}}}, {\color{blue} \frac{1}{5}}, {\color{blue} \frac{1}{3}},\mathbf{{\color{blue} \frac{3}{5}}}, {\color{blue} 1} \right\}.
\end{eqnarray*}

We have presented the slopes in color in order to represent those in, respectively, ${\color{red}  A_{\mbf{a-2}}}$, $\mathbf{\color{red} A_\mbfa \setminus A_{\mbf{a-2}}}$, ${\color{blue} B_{\mbf{a-2}}}$, and $\mathbf{{\color{blue} B_\mbfa \setminus B_{\mbf{a-2}}}}$. Note that the elements of $A_\mbfa\setminus A_{\mbf{a-2}}$ (resp., $B_\mbfa \setminus B_{\mbf{a-2}}$) are bolded while the elements of $A_{\mbf{a-2}}$ (resp., $B_\mbf{a-2}$) are not bolded.  For example, the elements $\frac{1}{4}$, $\frac{1}{3}$ and $\frac{2}{3}$ in $\slopesa{5}$ listed above are bolded, while the elements $0$, $\frac{1}{2}$ and $1$ in $\slopesa{5}$ are not.  The distinction between slopes in $A_{\mbf{a-2}}$ and slopes in $A_\mbfa \setminus A_{\mbf{a-2}}$ will be crucial to our refinement of Theorem 4.1 of \cite{Du-CaTy} stated below as Theorem \ref{thm:refinement}, as well as our main results stated in Theorem \ref{thm:sequenceOfCompatiblePeriodicOrbits} and Proposition \ref{prop:compatibleclosedorbits}.


\subsection{Mathematical billiards}
\label{subsec:mathematicalBilliards}

Consider a point mass making a collision in a boundary subject to the Law of Reflection.  That is, the angle of reflection is equal to the angle of incidence.  Alternatively (and equivalently), one may consider the incoming vector reflected through the tangent at the point of collision; see Figure \ref{fig:billiardMap}. Furthermore, we require, prior to colliding in the boundary, that the point mass be traveling in a straight line. For the remainder of the article, we assume that a mathematical billiard table $\Omega(D)$ with boundary $D$ is a path-connected region in the plane with sufficiently piecewise smooth boundary\footnote{In the case of a polygonal billiard table, there are finitely many vertices where a well-defined tangent does not exist.} that is also simple and connected; see Figure \ref{fig:billiardEllipse} for an example of such a billiard table.

\begin{figure}
\begin{center}
\includegraphics[scale=.65]{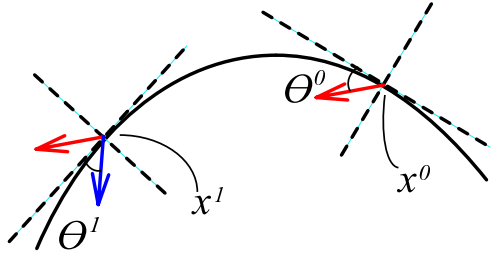}
\end{center}
\caption{A demonstration of how a billiard ball reflects in the sufficiently smooth boundary of a billiard table.}
\label{fig:billiardMap}
\end{figure}

\begin{figure}
\begin{center}
\includegraphics[scale=.5]{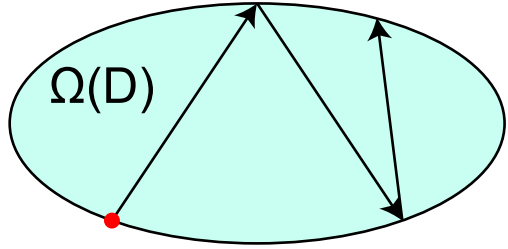}
\end{center}
\caption{The billiard table with a boundary that is an ellipse constitutes an example of a mathematical billiard.}
\label{fig:billiardEllipse}
\end{figure}

An initial condition of the billiard flow is given by $(x^0,\theta^0)$, where we take $x^0$ to be on the boundary $D$ of the billiard table $\Omega(D)$ and $\theta^0$ to be an angle measured relative to the tangent line at $x^0$.  The next point of collision $x^1$ in the boundary is determined by the billiard map $f_D$. That is, $f_D(x^0,\theta^0) = (x^1,\theta^1)$, where $\theta^1$ is an inward pointing vector based at $x^1$. Then, successive iterates $f^k$, $k \geq 1$, of the billiard map $f$ determine the collision point $x^k$ of the billiard ball in the boundary $D$ of the billiard table. To be clear, the angle of reflection $\theta^k$, determined from the Law of Reflection, is the angle made by the vector based at $x^k$ pointing inward towards the interior of $\Omega(D)$, measured relative to the tangent based at $x^k$.  Then, one can describe an equivalence relation $\sim$ on $\Omega(D)\times S^1$ that amounts to identifying the outward pointing vector based at a point $x^k$ with the inward pointing vector based at $x^k$.\footnote{For a more complete treatment of the billiard map $f_D$, the billiard flow $\phi^t$ in the phase space and an equivalence relation placed on $\Omega(D)\times S^1$, see \cite{Sm}.}


Then $f_D([(x^0,\theta^0)]) = [(x^1,\theta^1)]$, where $\theta^1$ is the angle made by the inward pointing vector, measured relative to the tangent at $x^1$.  For now, we shall denote by $(x,\theta)$ the equivalence class of $[(x,\theta)]$, relative to the equivalence relation $\sim$.  That is, $(x,\theta)$, such that $\theta$ is the angle determined from the inward pointing vector, is the representative element of the equivalence class $[(x,\theta)]$. In order to reduce the complexity of the phase space, we only consider the space $(D\times S^1)/\sim$.

\begin{definition}[An orbit of $\Omega(D)$]
Let $f_D$ be the billiard map describing the discrete billiard flow in the phase space $(\Omega(D)\times S^1)/\sim$.  An orbit $\mathscr{O}(x^0,\theta^0)$ of $\Omega(D)$ is then given by $\{f_D^n(x^0,\theta^0)\}_{n=0}^\infty$.
\end{definition}

We frequently make an abuse of notation and say that an orbit is the path traversed by the billiard ball connecting the base points $x^i$ of $f_D^i(x^0,\theta^0)$, $0\leq i\leq N$.

\begin{definition}[A closed orbit of $\Omega(D)$]
An orbit $\mathscr{O}(x^0,\theta^0)$ of $\Omega(D)$ is said to be \textit{closed} if the orbit consists of finitely many elements.
\end{definition}

\begin{definition}[A periodic orbit of $\Omega(D)$]
Let $\mathscr{O}(x^0,\theta^0)$ be a closed orbit of $\Omega(D)$.  If there exists a least integer $m\geq 1$ such that $f^m(x^0,\theta^0) = (x^0,\theta^0)$, then we say the orbit $\mathscr{O}(x^0,\theta^0)$ is a \textit{periodic} orbit.

\end{definition}

\begin{remark}
One does not call an orbit \textit{closed} simply because one decides to arbitrarily terminate the billiard flow. If, under the billiard flow, the billiard ball intersects a point of the boundary for which reflection cannot be determined in a well-defined manner, then the billiard ball trajectory terminates at that point.  Such an orbit is then called \emph{singular}. If $\mathscr{O}(x^0,\theta^0)$ is a singular orbit and there exists $k\in \mathbb{N}$ such that the base point $x^{-k}$ of $f^{-k}(x^0,\theta^0)$ is a point of $D$ not admitting a well defined derivative, then we say $\mathscr{O}(x^0,\theta^0)$ forms a \textit{saddle connection}.  In such a case, $\mathscr{O}(x^0,\theta^0)$ is also a closed orbit.
\end{remark}

\begin{definition}[A dense orbit of $\Omega(D)$]
If the path traversed by a billiard ball is dense in $\Omega(D)$ (forward or backward in time), then the orbit $\mathscr{O}(x^0,\theta^0)$ is said to be dense in $\Omega(D)$.
\end{definition}

\begin{definition}[Rational polygonal billiard]
A \textit{rational polygon} is a polygon where each interior angle $v_i$ is of the form $\frac{p_i}{q_i}\pi$, $p_i,q_i\in \Z$, $q_i\neq 0$. A \textit{rational polygonal billiard} is a billiard table with boundary $D$, where $D$ is a rational polygon.
\end{definition}

\subsubsection{Translation surfaces via rational polygonal billiard tables}
\label{subsubsec:TranslationSurfacesViaRationalPolygonalBilliardTables}

Consider a rational polygonal billiard $\Omega(D)$ with interior angles $\frac{p_1}{q_1}\pi,...,\frac{p_k}{q_k}\pi$.  If $N=\text{lcm}(q_1,...,q_k)$, then, by appropriately identifying the sides of $2N$ many copies of $\Omega(D)$, we can construct a compact, connected orientable surface with finitely many conic singularities, with the coordinate changing functions given by translations.\footnote{For a more complete definition of translation surface as well as descriptions of closely related types of surfaces, please see \cite{MasTa,Gut2,HuSc,Zo} and the relevant references therein.}  In addition, if the conic angle about a conic singularity is $2\pi$, then the singularity is called a \textit{removable singularity}; if the conic angle is greater than $2\pi$, then the singularity is called a \textit{non-removable singularity}.

Consider the unit square $Q$.  Then, $N=\text{lcm}(2,2,2,2) = 2$.  In Figure \ref{fig:fourSquareSurface}, we indicate how the sides of $2N=4$ many copies of $\Omega(Q)$ must be identified so as to produce a translation surface. In this example, each corner is identified with three other corners to produce a singularity of the translation surface.  In each case, such a singularity has a conic angle of $2\pi$. This implies that the billiard flow at each vertex can be made well defined. We note that, in this case of the square, the translation surface is topologically equivalent to the flat torus.

\begin{figure}
\begin{center}
\includegraphics[scale=0.5]{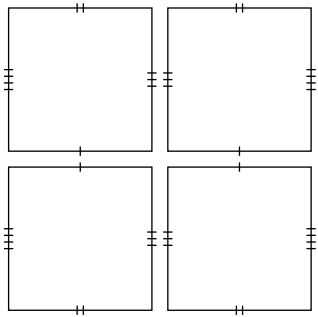}
\end{center}
\caption{A translation surface constructed from the unit square billiard $\Omega(R)$.  Opposite and parallel sides are identified.}
\label{fig:fourSquareSurface}
\end{figure}

If $\{u_1,u_2\}$ is a basis for $\mathbb{R}^2$, then a vector $z\in \mathbb{R}^2$  is called \textit{rational with respect to} $\{u_1,u_2\}$ if $z=mu_1+nu_2$, for some $m,n\in \mathbb{Z}$. Combining the results of \cite{GutJu1} with Theorem 3 of \cite{Gut2}, we can state the following result, which we do not claim as a new theorem, but which we rephrase in a way that is suitable for our purposes.\footnote{This same statement is also given in \cite{LapNie3} and is similarly attributed to E. Gutkin.}

\begin{theorem}[{\cite{Gut2}}]
\label{thm:gutkinsResult}
Let $\mathcal{S}(D)$ be a translation surface determined from a rational polygonal billiard $\Omega(D)$.  If $\mathcal{S}(D)$ is a branched cover of a singly punctured torus, then a geodesic on $\mathcal{S}(D)$ is periodic or forms a saddle connection if and only if the geodesic has an initial direction that is rational.  In addition, a geodesic on $\mathcal{S}(D)$ is dense if and only if the geodesic has an initial direction that is irrational.
\end{theorem}

Theorem \ref{thm:gutkinsResult} is equivalent to the following, similarly attributed to E. Gutkin.

\begin{theorem}[{\cite{Gut2}}]
\label{thm:gutkinsResultOnBilliards}
Let $\Omega(D)$ be a rational billiard table that is tiled by an integrable billiard\footnote{If $\Omega(D)$ is a billiard table that tiles the plane, then $\Omega(D)$ is called an \textit{integrable billiard table}.} table $\Omega(P)$. Then an orbit on $\Omega(D)$ is closed if and only if the orbit has an initial direction that is rational.  In addition, an orbit on $\mathcal{S}(D)$ is dense if and only if the orbit has an initial direction that is irrational.
\end{theorem}

If $\Omega(D)$ is a rational polygon with $k$ many sides and interior angles $\frac{p_i}{q_i}\pi$, $i\leq k$ and $N=\text{lcm}\{q_1,q_2,...,q_k\}$, then the genus $g$ of a translation surface constructed from a rational polygonal billiard $\Omega(D)$ is given by\footnote{See \cite{MasTa} for an explanation of the formula describing the genus of a translation surface determined from a rational billiard table $\Omega(D)$.}

\begin{align}
\label{eqn:genus} g&=1+\frac{N}{2}\left(k-2-\sum_{i=1}^k\frac{1}{n_i}\right).
\end{align}

\subsubsection{Unfolding a billiard orbit and equivalence of flows}
\label{subsubsec:UnfoldingABilliardOrbit}
Consider a rational polygonal billiard $\Omega(D)$ and an orbit $\mathscr{O}(x^0,\theta^0)$ of $\Omega(D)$.  Reflecting the billiard $\Omega(D)$ and the orbit in the side of the billiard table containing the base point $x^1$ of the orbit ($x^1$ being the first point of collision after starting from $x^0$) partially unfolds the orbit $\mathscr{O}(x^0,\theta^0)$; see Figure \ref{fig:PartiallyUnfoldingAnOrbitOfTheSquareBilliard} for the case of the square billiard. Continuing this process until the orbit is a straight line produces as many copies of the billiard table as there are base points of the orbit; see Figure \ref{fig:UnfoldingAnOrbitOfTheSquareBilliard}. That is, if the period of an orbit $\mathscr{O}(x^0,\theta^0)$ is some positive integer $p$, then the number of copies of the billiard table in the unfolding is also $p$. We refer to such a straight line as the \textit{unfolding of the billiard orbit}.

\begin{figure}
\begin{center}
\includegraphics[scale=.35]{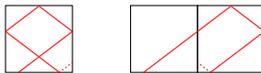}
\end{center}
\caption{Partially unfolding an orbit of the square billiard $\Omega(Q)$.}
\label{fig:PartiallyUnfoldingAnOrbitOfTheSquareBilliard}
\end{figure}

\begin{figure}
\begin{center}
\includegraphics[scale=.35]{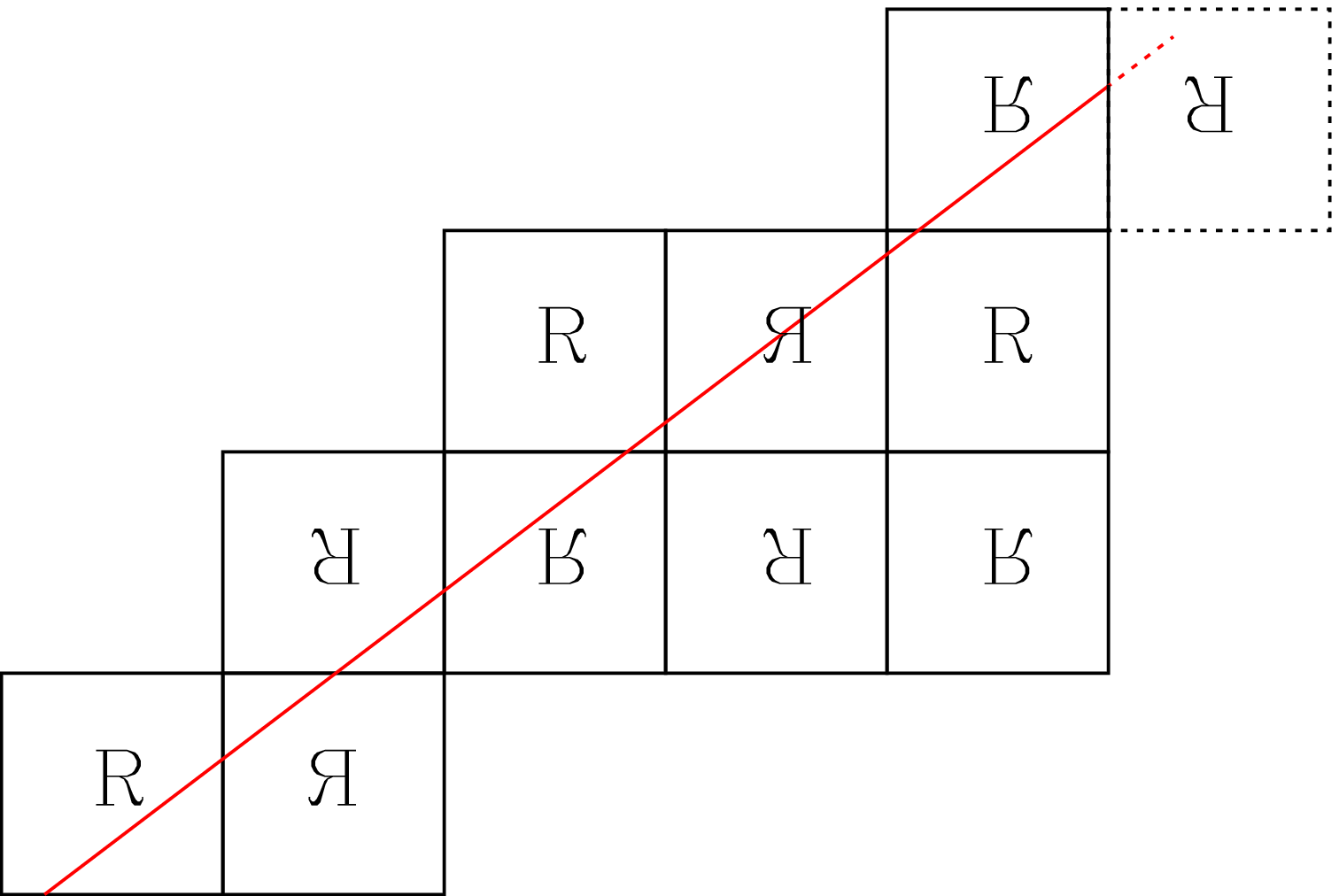}
\end{center}
\caption{Unfolding an orbit of the square billiard $\Omega(Q)$.}
\label{fig:UnfoldingAnOrbitOfTheSquareBilliard}
\end{figure}

We note that the construction of a translation surface from a rational billiard $\Omega(D)$ amounts to letting a group of symmetries act on $\Omega(D)$; see \S\ref{subsubsec:TranslationSurfacesViaRationalPolygonalBilliardTables}.  That is, a rational billiard $\Omega(D)$ can be acted on by a dihedral group $\mathscr{D}_N$ to produce a translation surface in a way that is similar to unfolding the billiard table.  Hence, we can quickly see how the billiard flow is dynamically equivalent to the geodesic flow on the associated translation surface; see Figure \ref{fig:equivalenceOfGeodesicFlow} and the corresponding caption.

\begin{figure}
\begin{center}
\includegraphics[scale=.35]{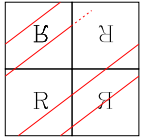}
\end{center}
\caption{Rearranging the unfolded copies of the unit square from Figure \ref{fig:UnfoldingAnOrbitOfTheSquareBilliard} and correctly identifying sides so as to recover the flat torus, we see that the unfolded orbit corresponds to a closed geodesic on the translation surface.}
\label{fig:equivalenceOfGeodesicFlow}
\end{figure}

One may modify the notion of ``reflecting'' so as to determine orbits of a billiard table tiled by a rational polygonal billiard $\Omega(D)$.  As an example, we consider the unit-square billiard table. An appropriately scaled copy of the unit-square billiard table can be tiled by the unit-square billiard table by making successive reflections in the sides of the unit square.  One may then unfold an orbit of the unit-square billiard table into a larger square billiard table. When the unfolded orbit of the original unit-square billiard intersects the boundary of the appropriately scaled (and larger) square, then one continues unfolding the billiard orbit in the direction determined by the Law of Reflection.\footnote{That is, assuming the unfolded orbit is long enough to reach a side of the larger square.}  We will refer to such an unfolding as a \textit{reflected-unfolding}.

We may continue this process in order to form an orbit of a larger scaled square billiard table. Suppose that an orbit $\mathscr{O}(x^0,\theta^0)$ has period $p$.  If $s$ is a positive integer, then we say that $\mathscr{O}^s(x^0,\theta^0)$ is an orbit that traverses the same path as $\mathscr{O}(x^0,\theta^0)$ $s$-many times.  For sufficiently large $s\in \mathbb{N}$, an orbit that traverses the same path as an orbit $\mathscr{O}(x^0,\theta^0)$ $s$-many times can be reflected-unfolded in an appropriately scaled square billiard table to form an orbit of the larger billiard table; see Figure \ref{fig:aReflectedUnfoldingInTheSquare}.  Such a tool is useful in understanding the relationship between the billiard flow on a rational polygonal billiard $\Omega(D)$ and a billiard table tiled by $\Omega(D)$, and will be particularly useful in understanding the material presented in the sequel.

\begin{figure}
\begin{center}
\includegraphics[scale=.35]{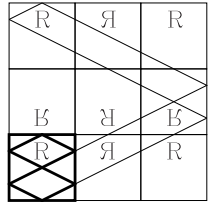}
\end{center}
\caption{Unfolding the orbit of the unit-square billiard in a (larger) scaled copy of the unit-square billiard. This constitutes an example of a reflected-unfolding of an orbit. The edges of the original unit-square billiard table and the segments comprising the orbit have been thickened to provide the reader with a clearer view of what is the orbit of the unit square and what constitutes the reflected-unfolding of the orbit of the larger square.}
\label{fig:aReflectedUnfoldingInTheSquare}
\end{figure}

\subsubsection{The billiard table $\omegasai{\mbfa}{n}$}
We note that $S_{\mbfa,n}$ does not have a connected boundary.  However, the boundary of $S_0$ and the boundary of each omitted square is connected.  In addition, the notation ``$\Omega(D)$'' indicates that $D$ is the boundary of the billiard table.  To be technically correct, $\Omega(\partial S_{\mbfa,n})$ is the proper way of referring to the prefractal billiard table.  In order to simplify notation, we refer to the prefractal billiard table by $\omegasi{n}$.



Clearly, each prefractal billiard $\omegasai{\mbfa}{n}$ can be interpreted as a square billiard with open subsquares removed.  In fact, every prefractal billiard $\omegasai{\mbfa}{n}$ is a rational billiard table.


\begin{notation}
\label{nota:AlphaForTheta}
Due to the fact that Theorem 4.1 in \cite{Du-CaTy} refers to the slope of a nontrivial line segment and we frequently refer to the main result of \cite{Du-CaTy}, we will denote the initial condition $(\xio{n},\theta_n^0)$ of an orbit of $\omegasai{\mbfa}{n}$ by $(\xio{n},\alphaio{n})$, where $\alphaio{n} = \tan(\theta_n^0)$.
\end{notation}

\begin{definition}[An orbit of the cell $\celli{k}{a^k}$ of $\omegasi{k}$]
 \label{def:AnOrbitOfACellOfOmegaSi}
 Consider the boundary of a cell $\celli{k}{a^k}$ of $\omegasi{k}$ as a barrier to the billiard flow.\footnote{Here, $C_{k,a^k}$ is a cell of the $k$th prefractal approximation $\sai{k}$, as given in Definition \ref{def:ACellOfSai} with each  $a_j$ equal to $a$.}  Then an orbit determined by reflecting in the boundary of the cell and with an initial condition contained in the cell is called an \textit{orbit of the cell $\celli{k}{a^k}$ of $\omegasi{k}$}.
\end{definition}

\begin{example}
In Figure \ref{fig:anExampleOfACell}, we saw an example of a cell of $S_{\mbf{7},1}$.  In Figure \ref{fig:anExampleOfAnOrbitOfACell}, we see an example of an orbit of the same cell.  The orbit of the cell shown has an initial condition of $((0,0),\frac{2}{3})$.
\end{example}

\begin{figure}
\includegraphics{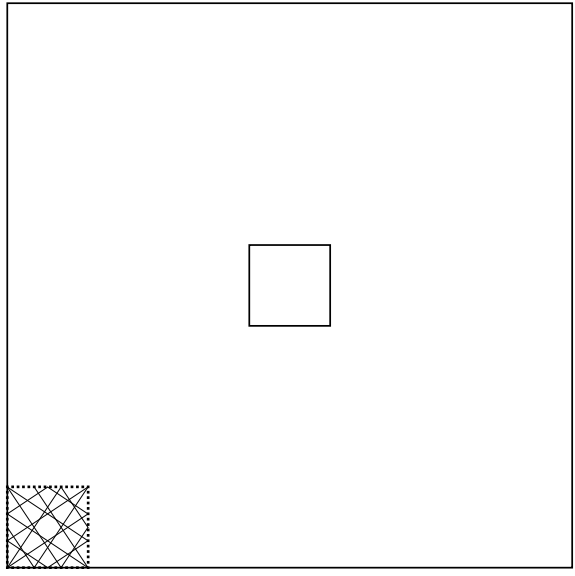}
\caption{An example of an orbit of a cell of $S_{\mbf{7},1}$.}
\label{fig:anExampleOfAnOrbitOfACell}
\end{figure}

\begin{remark}
So as to be clear, the boundary of the cell $C_{k,a^k}$ does not form a barrier to the billiard flow on $\Omega(S_{\mbfa,k})$.  Rather, we are treating the cell $\celli{k}{a^k}$ as a billiard table in its own right, embedded in the larger prefractal approximation $\omegasi{k}$.  Our motivation for doing so is found in the fact that we can proceed to reflect-unfold an orbit of a cell in $\omegasi{k}$.
\end{remark}

Recall from Definition \ref{def:ASelfSimSierpinskiCarpetViaAnIFS} that a self-similar Sierpinski carpet $S_\mbfa$ is the unique fixed point attractor of a suitably chosen iterated function system $\{\phi_{j}\}_{j=1}^{a^2-1}$ consisting of similarity contractions.  In light of this, an orbit of a cell $\celli{k}{a^k}$ of $\omegasi{k}$ is the image of an orbit $\osi{0}$ of the unit-square billiard $\Omega(S_0)$ under the action of a composition of contraction mappings $\phi_{m_k}\circ\cdots\circ\phi_{m_1}$, with $1\leq m_i\leq a^2-1$ and $1\leq i\leq k$, determined from the iterated function system $\{\phi_j\}_{j=1}^{a^2-1}$ for which $S_\mbfa$ is the unique fixed point attractor.

\subsection{Fractal billiards}
\label{subsec:fractalBilliards}
We recall various definitions from \cite{LapNie3}.  So that the reader may find the definitions more accessible, we phrase the following in terms of a self-similar Sierpinski carpet $S_\mbfa$ and a prefractal approximation $S_{\mbfa,n}$.

\begin{figure}
\begin{center}
\includegraphics[scale=.6]{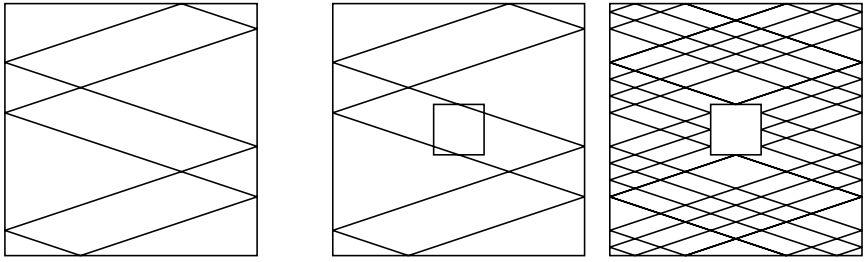}
\end{center}
\caption{In the first image on the left, we see an orbit of $\Omega(S_0)$ that has an initial condition of  $((\frac{3}{10},0),\frac{1}{3})$.  In the second image, we see that the same orbit would intersect the omitted square in the first level approximation of $\omegasa{\mbf{5}}$.  The third image is an orbit of $\Omega(S_{\mbf{5},1})$ with the same initial condition as the orbit shown in the first image.}
\label{fig:SequenceOfCompatibleOrbits}
\end{figure}


\begin{definition}[Compatible initial conditions]
\label{def:compatibleInitialConditions}
Without loss of generality, suppose that $n$ and $m$ are nonnegative integers such that $n > m$. Let $(\xio{n},\alpha_n^0)\in (\omegasai{\mbfa}{n}\times S^1)/\sim$ and $(\xio{m},\alpha_m^0)\in (\omegasai{\mbfa}{m}\times S^1)/\sim$ be two initial conditions of the orbits $\ofraciang{n}{\alpha_n^0}$ and $\ofraciang{m}{\alpha_m^0}$, respectively, where we are assuming that $\alpha_n^0$ and $\alpha_m^0$ are both inward pointing.  If $\alpha_n^0 = \alpha_m^0$ and if $\xio{n}$ and $\xio{m}$ lie on a segment determined from $\alpha_n^0$ (or $\alpha_m^0$) that intersects $\partial\sai{n}$ only at $\xio{n}$, then we say that $(\xio{n},\alpha_n^0)$ and $(\xio{m},\alpha_m^0)$ are \textit{compatible initial conditions}.
\end{definition}

\begin{remark}
\label{rmk:WhenAreOrbitsCompatible}
When two initial conditions $(\xio{n},\alpha_n^0)$ and $(\xio{m},\alpha_m^0)$ are compatible, then we simply write each as $(\xio{n},\alpha^0)$ and $(\xio{m},\alpha^0)$. If two orbits $\ofraciang{m}{\alpha_m^0}$ and $\ofraciang{n}{\alpha_n^0}$ have compatible initial conditions, then we say such orbits are \textit{compatible}.
\end{remark}

It may be the case that an initial condition $(\xio{n},\alpha^0_n)$ is not compatible with $(\xio{m},\alpha^0_m)$, for any $m<n$.  As such, in Definitions \ref{def:sequenceOfCompatibleInitialConditions} and \ref{def:SequenceOfCompatibleOrbits}, we consider sequences beginning at $i=N$, for some $N\geq 0$; see Figure \ref{fig:notCompatibleInitialConditions} and the corresponding caption.

\begin{figure}
\begin{center}
\includegraphics[scale=.55]{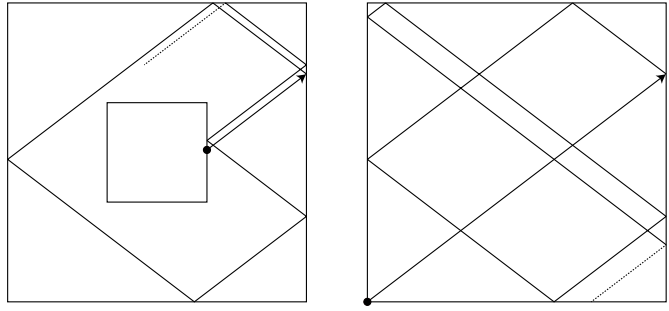}
\end{center}
\caption{In this figure, we see on the left an orbit of $\omegasai{\mbf{3}}{1}$ beginning at a point on the side of the peripheral square indicated by the small disc.  On the right, we see an orbit of the square billiard $\Omega(S_0)$ beginning at the origin $(0,0)$ (again, indicated by a small disc) with the same initial direction as the orbit on the left.  While both initial base points are collinear in the direction dictated by the initial direction of either orbit, these two orbits are not compatible because their initial conditions are not compatible.  Specifically, the line connecting the two initial base points intersects an additional point of the boundary of the billiard table, thereby preventing the initial conditions from being compatible initial conditions.}
\label{fig:notCompatibleInitialConditions}
\end{figure}

\begin{definition}[Sequence of compatible initial conditions]
\label{def:sequenceOfCompatibleInitialConditions}
Let $\{(\xio{i},\alpha_i^0)\}_{i=N}^\infty$ be a sequence of initial conditions, for some integer $N\geq 0$.  We say that this sequence is a \textit{sequence of compatible initial conditions} if for every $m\geq N$ and for every $n> m$, we have that $(\xio{n},\alpha_n^0)$ and $(\xio{m},\alpha_m^0)$ are compatible initial conditions.  In such a case, we then write the sequence as $\{(\xio{i},\alpha^0)\}_{i=N}^\infty$.
\end{definition}

\begin{definition}[Sequence of compatible orbits]
\label{def:SequenceOfCompatibleOrbits}
Consider a sequence of compatible initial conditions $\{(\xio{n},\alpha^0)\}_{n=N}^\infty$.  Then the corresponding sequence of orbits $\compseqsi{N}$ is called \textit{a sequence of compatible orbits}.
\end{definition}

If $\ofracixang{m}{\xio{m}}{\alpha_m^0}$ is an orbit of $\omegasi{m}$, then $\ofracixang{m}{\xio{m}}{\alpha_m^0}$ is a member of a sequence of compatible orbits $\compseqsiang{N}{\alpha^0}$ for some $N\geq 0$. It is clear from the definition of a sequence of compatible orbits that such a sequence is uniquely determined by the first orbit $\ofraciang{N}{\alpha^0}$.  Since the initial condition of an orbit determines the orbit, we can say without any ambiguity that a sequence of compatible orbits is determined by an initial condition $(\xio{N},\alpha^0)$.

\begin{definition}[A sequence of compatible $\mathcal{P}$ orbits]
Let $\mathcal{P}$ be a property (resp., $\mathcal{P}_1,...,\mathcal{P}_j$ a list of properties).  If every orbit in a sequence of compatible orbits has the property $\mathcal{P}$ (resp., a list of properties $\mathcal{P}_1,...,\mathcal{P}_j$), then we call such a sequence \textit{a sequence of compatible $\mathcal{P}$ \emph{(}resp., $\mathcal{P}_1,...,\mathcal{P}_j$\emph{)}  orbits}.
\end{definition}

Considering the fact that $\Omega(S_{\mbfa,n})$ is tiled by a square (i.e., an integrable billiard), Theorems \ref{thm:gutkinsResult} and \ref{thm:gutkinsResultOnBilliards} allow us to construct sequences of compatible closed orbits and sequences of compatible dense orbits.  Furthermore, under the right conditions, we will be able to construct sequences of compatible periodic orbits. That is, we will be able to construct sequences of compatible orbits in which each orbit is a nonsingular, closed orbit in its respective billiard table $\Omega(S_{\mbfa,n})$.

\section{A refinement of Theorem 4.1 in \cite{Du-CaTy}}
\label{sec:ARefinementOfTheorem4.1OfDuCaTy}
We now discuss Theorem 4.1 in \cite{Du-CaTy} and provide various examples indicating the necessity for refining the latter half of their statement.



In Part 1a of the proof of Theorem 4.1 of \cite{Du-CaTy}, it is determined that nontrivial line segments with slope $\alpha\in A$ necessarily avoid omitted squares of $S_\mbfa$ (recall that the omitted squares are the open squares removed in the construction of $S_\mbfa$).  However, in stating their result, the authors of \cite{Du-CaTy} say that such nontrivial line segments necessarily intersect corners of peripheral squares, which is more restrictive than what is concluded in the proof.  We provide an example where a nontrivial line segment beginning at $(0,0)$ (in accordance with their proof) with a slope $\alpha\in A$ does not intersect any corner of any peripheral square.  Furthermore, in the proof of Theorem 4.1 in \cite{Du-CaTy}, the authors assume in Part 1a that a nontrivial line segment with slope $\alpha\in A$ starts from the origin $(0,0)$ and in Part 1b that a nontrivial line segment with slope $\alpha \in B$ starts from $(\frac{1}{2},0)$, but they do not state this in the assumptions of their theorem. This does not invalidate their characterization of self-similar Sierpinski carpets, but can be included in the statement of the Theorem 4.1 in \cite{Du-CaTy} and further expanded upon.  As our examples will show, such assumptions are necessary to state and can actually be made more robust to allow for an explicit description of other nontrivial line segments (some of which are not maximal,\footnote{A maximal nontrivial line segment is a nontrivial line segment that connects two segments of the boundary of $S_\mbfa$ corresponding to the boundary unit square $S_0$.} but are nontrivial nonetheless).

In Remark 4.3 of \cite{Du-CaTy}, the authors indicate that one can translate the initial base point of a nontrivial line segment however one chooses (so long as the translation of the base point remains within the base of the unit square and not outside the unit square) and still maintain that the line segment starting from the new base point and having the same slope as before remains as a nontrivial line segment of $S_\mbfa$. We provide two examples for which this is not the case.

We summarize the above discussion by listing the examples indicating that the latter half of the statement in Theorem 4.1 in \cite{Du-CaTy} must be refined or stated separately as an additional theorem. These examples will be discussed in greater detail in \S \ref{sec:EventuallyConstantSequencesOfCompatiblePeriodicOrbitsOfSa}.

\begin{enumerate}
\item{A nontrivial line segment of $S_\mbf{7}$ with slope $\alpha=\frac{2}{3}\in A$ beginning at $(0,0)$ that avoids corners of every peripheral square.}
\item{A nontrivial line segment of $S_\mbf{7}$ with slope $\alpha=\frac{2}{3}\in A$ beginning at $(\frac{1}{2},0)$ that avoids corners of every peripheral square.}
\item{A nontrivial line segment of $S_\mbf{5}$ with slope $\alpha=\frac{1}{3}\in B$ beginning from $(\frac{1}{2},0)$ that, when translated to the origin $(0,0)$, constitutes a trivial line segment of $S_\mbf{5}$.}
\item{A nontrivial line segment of $S_\mbf{7}$ with slope $\alpha=\frac{3}{4}\in A$ beginning from $(0,0)$ that, when translated to $(\frac{1}{2},0)$, constitutes a trivial line segment of $S_\mbf{7}$.}
\end{enumerate}

\begin{figure}
\begin{center}
\includegraphics[scale=0.8]{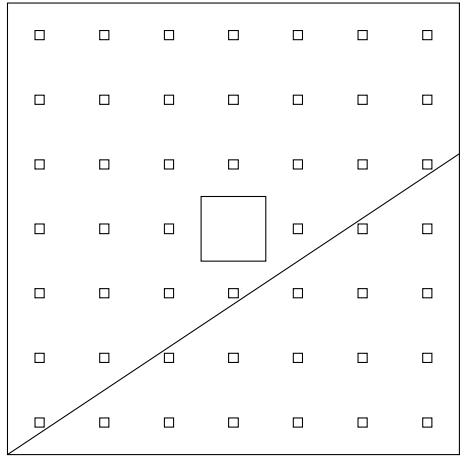}
\end{center}
\caption{In this figure, we see a nontrivial line segment beginning from the origin $(0,0)$ and having a slope of $\frac{2}{3}$.  Such a segment will avoid every peripheral square of $S_{\mbf{7}}$.}
\label{fig:2-3NTLSStartingFrom0}
\end{figure}

 The above examples indicate the necessity to distinguish between slopes in $A_{\mbf{a-2}}$ and slopes in $A_\mbfa \backslash A_{\mbf{a-2}}$. This is the key point of our Theorem \ref{thm:refinement}.

In preparation for our result, we give three lemmas.

\begin{lemma}
\label{lem:AaNotAa-2}
Let $S_\mbfa$ be a self-similar Sierpinski carpet.  Then,
\begin{align}
\notag A_\mbfa \setminus A_{\mbf{a-2}} &= \left\{\frac{p}{q} : 1\leq p < q, ~ p+q=a,~ p \text{ and } q \text{ are coprime} \right\}.
\end{align}
\end{lemma}

\begin{proof}
We first note that if $\alpha\in A_\mbfa$, then $\alpha=\frac{p}{q}$ such that
\begin{align}
\notag p+q \leq a & \text{ and } 0\leq p< q\leq a-1.
\end{align}
\noindent Likewise, if $\frac{p}{q}\in A_{\mbf{a-2}}$, then
\begin{align}
\notag p+q\leq a-2 & \text{ and } 0\leq p < q.
\end{align}
Hence, $\frac{p}{q}\in A_\mbfa \setminus A_{\mbf{a-2}}$ satisfies either
\begin{align}
\label{eqn:Condition1}a-2<p+q \leq a &\text{ and } 1\leq p<q\leq a-1,
\end{align}
or
\begin{align}
\label{eqn:Condition2} p+q\leq a, \, 1\leq p &< q \leq a-1 \text{ and }  q>a-3,
\end{align}
\noindent in addition to $p+q$ being odd \emph{\textbf{and}} $p$, $q$ being coprime.  It can be deduced from the inequalities given on line (\ref{eqn:Condition1}) that $p+q = a$. Similarly, since $p$ and $q$ are nonnegative integers, it can be deduced from the inequalities on line (\ref{eqn:Condition2}) that $q=a-2$ or $q=a-1$, which forces $p=2$ or $p=1$, respectively. (Otherwise, $p+q$ would be even or exceed the value $a$.) The union of these two sets comprises all $\frac{p}{q}$ for which $p+q=a$, $1\leq p <q$, and $p$ and $q$ being coprime.
\end{proof}

\begin{lemma}
Let $S_\mbfa$ be a self-similar Sierpinski carpet.  If $\alpha\in B_\mbfa$, $n\geq 1$, and $r$ an odd positive integer with $r\leq 2a^n$, then the line $y=\alpha(x-\frac{r}{2a^n})$ avoids points of the plane of the form $(\frac{u}{a^l},\frac{v}{a^m})$, $u,v,l,m\in \mathbb{Z}$.
\label{lem:BAvoids}
\end{lemma}

\begin{proof}
Let $S_\mbfa$ be a self-similar Sierpinski carpet billiard, $\alpha\in B_\mbfa$, $n\geq 1$, and $r$ a positive odd integer with $r\leq 2a^n$.  Suppose there exist $u,v,l,m\in \mathbb{Z}$, $l,m\geq 0$ such that $y=\alpha(x-\frac{r}{2a^n})$ intersects the point $(\frac{u}{a^l},\frac{v}{a^m})$.  Then,

\begin{align}
\frac{v}{a^m} &= \alpha\left(\frac{u}{a^l}-\frac{r}{2a^n}\right)\\
2va^{n+l} &= \alpha a^m(2ua^n-ra^l).
\end{align}

\noindent Since $\alpha = \frac{p}{q}$, $p,q\in \mathbb{Z}$, $p$ and $q$ both odd and positive, we have that

\begin{align}
2qva^{n+l} &= p a^m(2ua^n-ra^l).\label{eqn:oddNotEqualToEvenAlphaInB}
\end{align}

The left-hand side of Equation (\ref{eqn:oddNotEqualToEvenAlphaInB}) is even and the right-hand side is odd, so we have an immediate contradiction.  Therefore, no line of the form $y=\alpha(x-\frac{r}{2a^n})$ will ever intersect any point of the form $(\frac{u}{a^l},\frac{v}{a^m})$, $u,v,l,m\in \mathbb{Z}$, $m,l\geq 0$.
\end{proof}

\begin{lemma}
Let $b,c\in \mathbb{N}$, $d=\gcd(b,c)$, and $k \in \mathbb{Z}$. If $d$ divides $k$, then the equation $bx+cy=k$ has integer solutions $(x,y)$. If $(x_0,y_0)$ is one such integer solution, then all integer solutions can be expressed in the form
\begin{equation}
(x,y)=\left(x_0+ m\frac{c}{d},~ y_0 - m\frac{b}{d}\right)
\end{equation}
for some $m\in\mathbb{Z}$.
\label{lem:divisor}
\end{lemma}

\begin{proof}
This is an application of the elementary fact (see \emph{e.g.} Theorem 1.3 of \cite{NiZuMo}) that there exist $x_1, y_1 \in \mathbb{Z}$ such that $\gcd(b,c)=b x_1 + c y_1$.
\end{proof}

We are ready to state our our refinement of Theorem 4.1 of \cite{Du-CaTy}.

\begin{theorem}[A refinement of {Theorem 4.1 in \cite{Du-CaTy}}]
Let $\mbfa=(\frac{1}{a},\frac{1}{a},\frac{1}{a},...)$ be a constant sequence.  Then the set of slopes $\slopesa{\mbfa}$ is the union of the following two sets\emph{:}

\begin{align}
A_\mbfa &= \left\{\frac{p}{q} \,\, :\,\, p+q\leq a,\, 0\leq p< q\leq a-1,\, p,q\in \N\cup \{0\},\, p+q \text{ is odd}\right\}, \label{eqn:Aaset}\\
B_\mbfa &= \left\{\frac{p}{q} \,\,:\,\, p+q\leq a-1,\, 0\leq p\leq q\leq a-2,\, p,q\in \N,\, p \text{ and } q \text{ are odd}\right\}. \label{eqn:Baset}
\end{align}

Moreover, if $\alpha\in A_{\mbfa} \setminus A_{\mbf{a-2}}$, then each nontrivial line segment in $S_\mbfa$ with slope $\alpha$ beginning from $(0,0)$ touches vertices of peripheral squares, while if $\alpha\in A_\mbf{a-2}$, then each nontrivial line segment in $S_\mbfa$ with slope $\alpha$ beginning from $(0,0)$ is disjoint from all peripheral squares. If $\alpha\in B_\mbfa$, then each nontrivial line segment in $S_\mbfa$ with slope $\alpha$ beginning at $(\frac{r}{2a^n},0)$, $n\geq 1$ and $r<2a^n$ being a positive odd integer, is disjoint from all peripheral squares with side-length $<\frac{1}{a^n}$.

In addition, for each $\alpha\in A_\mbfa\cup B_\mbfa$, there exist maximal line segments in $S_\mbfa$ with slope $\alpha$.  Finally, if $b<a$, then any maximal nontrivial line segment in $S_\mbf{b}$ is also contained in $S_\mbfa$.  In particular, $\slopesa{\mbf{b}}\subseteq\slopesa{\mbfa}$.
\label{thm:refinement}
\end{theorem}

\begin{proof}
We know from the proof of Theorem 4.1 in \cite{Du-CaTy} that the set $\slopesa{\mbfa}$ can be partitioned as described in Equations (\ref{eqn:Aaset}) and (\ref{eqn:Baset}).  Moreover, the proof of Theorem 4.1 in \cite{Du-CaTy} is valid for the case when $\alpha\in B_\mbfa$ and the nontrivial line segment with slope $\alpha$ begins at $(\frac{1}{2},0)$; the proof also remains valid for when $\alpha\not\in\slopesa{\mbfa}$.  Hence, we focus on the cases where $\alpha\in A_\mbfa \setminus A_{\mbf{a-2}}$ and a nontrivial line segment with slope $\alpha$ begins at $(0,0)$; $\alpha\in A_{\mbf{a-2}}$ and a nontrivial line segment with slope $\alpha$ begins at $(0,0)$; and $\alpha\in B_{\mbfa}$ and a nontrivial line segment with slope $\alpha$ begins at $(\frac{r}{2a^n},0)$, for some $n\geq 1$, $r\neq a$ and $r$ an odd positive integer.

\textbf{Case 1: $\alpha \in A_\mbfa \setminus A_{\mbf{a-2}}$.} First note that the inferior right corner of a peripheral square of $S_\mbfa$ has coordinate
\begin{equation}
\left(\frac{a+1}{2a^{n+1}},\frac{a-1}{2a^{n+1}}\right) +\left(\frac{u}{a^n},\frac{v}{a^n}\right)
\label{eq:coordinfrightcorner}
\end{equation}
\noindent for some $u$,$v \in \mathbb{N}\cup \{0\}$, with $0\leq u,v\leq a^n-1$ and $u\neq v$.  We wish to show that  any line segment emanating from $(0,0)$ with slope $\alpha=\frac{p}{q} \in A_\mbfa \setminus A_{\mbf{a-2}}$ hits the inferior right corner of some peripheral square: that is, there exist $0\leq u,v\leq a^n-1$ such that the equation
\begin{equation}
\frac{1}{q} \left(u+ \frac{a+1}{2a}\right) = \frac{1}{p}\left(v+\frac{a-1}{2a}\right)
\label{eq:line}
\end{equation}
holds.

By Lemma \ref{lem:AaNotAa-2}, $p+q=a$ for any $\frac{p}{q} \in A_\mbfa \setminus A_{\mbf{a-2}}$, so Equation (\ref{eq:line}) can be rewritten as
\begin{align}
2(a-q) u - 2q v &= 2q-a-1.
\label{eq:line2}
\end{align}
Geometrically speaking, the next step  is to find, in the $uv$-plane (regarding $u$ and $v$ as real-valued rather than integer-valued), the intersections between the line defined by Equation (\ref{eq:line2}) and the lattice $\{(u,v):u,v\in \mathbb{Z} \cap [0,a^n-1], u\neq v\}$. Observe that the line defined by Equation (\ref{eq:line2}) has $v$-intercept $\frac{a+1}{2q}-1$, which is between $-\frac{1}{2} +\frac{1}{a-1}$ and $0$ (because $\frac{a+1}{2} \leq q \leq a-1$ by Lemma \ref{lem:AaNotAa-2}). If $a+1-2q=0$ (or $p=q-1$), then immediately $(u,v)=(0,0)$ is an intersection point. If not, we move along the line in integer increments of $u$, and check whether or not the $v$-coordinate is an integer. The successive $v$-coordinates can be written as
\begin{align}
v &= \left(\frac{a+1}{2q}-1\right) + u\left(\frac{a-q}{q}\right) \\
  &= \frac{(a+1)+2au}{2q} - (u+1),
\end{align}
\noindent with $u\in \mathbb{Z} \cap [0,a^n-1]$.  Thus it is enough to find a $u\in \mathbb{Z} \cap [0,a^n-1]$ such that $q$ divides $\left(\frac{a+1}{2} +au\right)$ (note that $a+1$ is even). An equivalent way to say this is that we're looking for integer solutions $(u,r)$ of the equation
\begin{equation}
au+qr = -\frac{a+1}{2}.
\label{eq:integersolutioneqn}
\end{equation}
Using that $\gcd(a,q) = \gcd(p,q)=1$ by Lemma \ref{lem:AaNotAa-2}, we apply Lemma \ref{lem:divisor} to deduce that there is a unique pair of integers  $(u^*,r^*)\in \{0,1,\cdots, q-1\} \times \mathbb{Z}$ which solves (\ref{eq:integersolutioneqn}): then $q$ divides $\left(\frac{a+1}{2} + au^*\right)$. Denote $v^* = \left(\frac{a+1}{2q}-1\right) + u^* \left(\frac{a-q}{q}\right)$ and note that $v^* < u^*$. It follows that the line segment emanating from $(0,0)$ with slope $\frac{a-q}{q} \in A_{\mbfa} \setminus A_{\mbf{a-2}}$ hits the inferior right corner of a peripheral square whose coordinate is $\frac{1}{a^n}(u^*+\frac{a+1}{2a} , v^*+\frac{a-1}{2a})$.

\textbf{Case 2: $\alpha \in A_{\mbf{a-2}}$.} Consider a nontrivial line segment with slope $\alpha$ beginning at $(0,0)$.  We suppose the nontrivial line segment intersects a peripheral square at a corner, whose coordinate is given by (\ref{eq:coordinfrightcorner}). Then, if $\alpha = \frac{p}{q} \in A_{\mbf{a-2}}$, we have

\begin{equation}
\frac{1}{q} \left(u+ \frac{a+1}{2a}\right) = \frac{1}{p}\left(v+\frac{a-1}{2a}\right)\!,
\end{equation}
which reduces to
\begin{equation}
\label{eqn:integerNotAFraction} q+2vq-p-2up =\frac{p+q}{a}.
\end{equation}

\noindent While the left-hand side of Equation (\ref{eqn:integerNotAFraction}) is an integer, the right-hand side is not, since $p+q\leq a-2 <a$.  This contradiction then implies that the nontrivial line segment cannot intersect any lower-right corner of any peripheral square (and, by symmetry, any upper-left corner of any peripheral square).

\textbf{Case 3: $\alpha \in B_\mbfa$.} Consider a nontrivial line segment of $S_\mbfa$ with slope $\alpha\in B_\mbfa$ emanating from $(\frac{1}{2},0)$, the existence of which follows from the proof of Theorem 4.1 in \cite{Du-CaTy}. We claim that a line segment emanating from $(\frac{r}{2a^n},0)$ with slope $\alpha$ will also be a nontrivial line segment of $S_\mbfa$ and never intersect the sides of any peripheral square with side-length $\frac{1}{a^k}$, $k\geq n$.

    To such end, consider $n\geq 0$, and $r$ an odd positive integer with $r\leq 2a^n$.  We scale the nontrivial line segment emanating from $(\frac{1}{2},0)$ by $\frac{1}{a^n}$. Then the scaled nontrivial line segment is a nontrivial line segment of the cell.  Then, there exists $(c\leq a^n,0)$ such that translating the line segment of the cell by $\frac{c}{a^n}$ results in a nontrivial line segment of a cell with $(\frac{r}{2a^n},0)$ as the midpoint of the base of the cell.  Moreover, the translated nontrivial line segment emanates from $(\frac{r}{2a^n},0)$ with slope $\alpha$.

    We now consider a tiling of the plane by the cell having $(\frac{r}{2a^n},0)$ as the midpoint of the base.  Then, a line $y=\alpha(x-\frac{r}{2a^n})$ avoids every peripheral square in the tiling.  If not, then the resulting line must also intersect a side of a peripheral square contained in the cell that does not correspond to a corner of the peripheral square by virtue of Lemma \ref{lem:BAvoids}.  Since $S_\mbfa$ is a self similar Sierpinski carpet, any line segment with slope $\alpha \in B_\mbfa$ beginning from $(\frac{r}{2a^n},0)$ must, by construction, be a trivial line segment in $S_\mbfa$.

\end{proof}

\section{Eventually constant sequences of compatible periodic orbits of $S_\mbfa$}
\label{sec:EventuallyConstantSequencesOfCompatiblePeriodicOrbitsOfSa}
In this section, we prove that the examples described in Items (1)--(4) in \S\ref{sec:ARefinementOfTheorem4.1OfDuCaTy} do in fact indicate a necessity for refining the latter half of the statement made in Theorem 4.1 in \cite{Du-CaTy}.  In so doing, we will provide proper motivation for our main results stated in Theorem \ref{thm:sequenceOfCompatiblePeriodicOrbits} and Proposition \ref{prop:compatibleclosedorbits}.  Finally, we will determine a family of periodic orbits of self-similar Sierpinski carpets.

We first note that the corners of the square billiard table constitute removable singularities of the billiard flow.  That is, reflection can be defined in the corners of the square in a well-defined manner.  Moreover, in every prefractal approximation of a Sierpinski carpet, reflection in the corners of the original square boundary can be made well-defined, as well.  Specifically, a billiard ball entering into a corner of the original square of any prefractal approximation of a Sierpinski carpet will exit the corner in a direction that makes an angle $\gamma$ that constitutes the reflection of the incoming angle $\theta$ as reflected through the angle bisector of the vertex; see Figure \ref{fig:ReflectionDefinedInCorners}.

\begin{figure}
\begin{center}
\includegraphics[scale=.35]{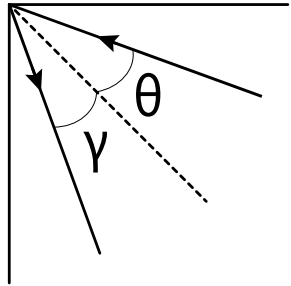}
\end{center}
\caption{The corners of a square billiard table constitute removable singularities of the billiard flow on the square billiard table $\Omega(S_0)$; see \S\ref{subsubsec:TranslationSurfacesViaRationalPolygonalBilliardTables} for a discussion of removable and nonremovable singularities of a translation surface.  Consequently, every corner of a prefractal approximation of a Sierpinski carpet $S_\mbfa$ corresponding to a corner of the original square constitutes a removable singularity of the billiard flow on $\omegasai{\mbfa}{n}$.}
\label{fig:ReflectionDefinedInCorners}
\end{figure}


In Item (1) of \S\ref{sec:ARefinementOfTheorem4.1OfDuCaTy}, we stated that a nontrivial line segment of $S_\mbf{7}$ with an initial starting point of $(0,0)$ and a slope $\alpha=\frac{2}{3}$ avoids corners of every peripheral square of $S_\mbf{7}$.  Consider an orbit $\osixang{0}{(0,0)}{\frac{2}{3}}$ of the square billiard table $\Omega(S_0)$.  We claim that this orbit avoids corners of every peripheral square of $\Omega(S_\mbf{7})$.


\begin{figure}
\begin{center}
\includegraphics{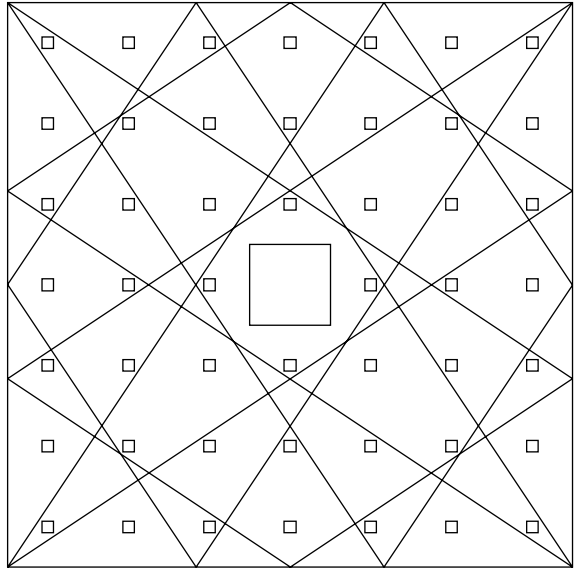}
\end{center}
\caption{In this figure, we see an orbit that avoids peripheral squares of $\Omega(S_{\mbf{7},2})$, the second level approximation of the self-similar Sierpinski carpet billiard table $\Omega(S_{\mbf{7}})$. As noted in the text, the fact that the orbit intersects corners does not pose a problem for determining the trajectory in a well-defined way.}
\label{fig:2-3orbitStartingFrom1-2}
\end{figure}

To such end, consider the orbit shown in Figure \ref{fig:2-3orbitStartingFrom1-2}, which has an initial condition $((0,0),\alpha)$, $\alpha = \frac{2}{3}$.  It is clear that this orbit avoids peripheral squares of $\Omega(S_{\mbf{7},2})$.  Then, scaling $\Omega(S_{\mbf{7},2})$ and the orbit together, relative to the origin, we see that $\frac{1}{7}\mathscr{O}((0,0),\alpha)$ avoids peripheral squares of the cell $\celli{1}{a}$ (i.e., the cell with side length $\frac{1}{7}$ and lower left corner corresponding to the origin); see Figure \ref{fig:2-3orbitStartingFrom1-2scaled}.  Not only does the orbit of the cell avoid the peripheral square with side-length $\frac{1}{49}$, it also avoids the peripheral squares with side length $\frac{1}{7^3}$, since the orbit $\mathscr{O}((0,0),\alpha)$ avoided every peripheral square of length $\frac{1}{49}$ in $\Omega(S_{\mbf{7},2})$.

\begin{figure}
\begin{center}
\includegraphics{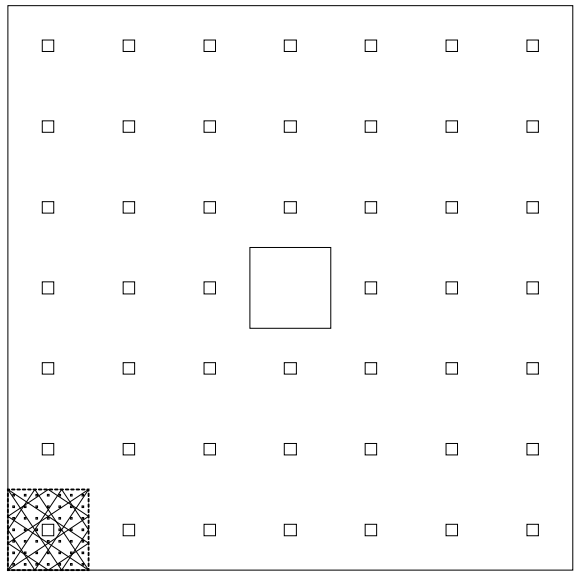}
\end{center}
\caption{We see here the billiard table $\Omega(S_{\mbf{7},2})$ and the orbit scaled by $\frac{1}{7}$. Such a scaled orbit is then an orbit of the cell $\celli{1}{a}$.  We then proceed to reflect-unfold this orbit to recover the orbit before scaling; see Figure \ref{fig:unfoldingTheScaledOrbit}.}
\label{fig:2-3orbitStartingFrom1-2scaled}
\end{figure}

Now, the reflected-unfolding of the orbit of the cell $\celli{1}{a}$ must follow the same path as $\mathscr{O}((0,0),\alpha)$ on account of both orbits having the same initial condition and reflection being an isometry; see Figure \ref{fig:unfoldingTheScaledOrbit}. This then means that $\mathscr{O}((0,0),\alpha)$ is an orbit of $\Omega(S_{\mbf{7},3})$ that avoids all peripheral squares of side-length $\frac{1}{7^3}$.

\begin{figure}
\begin{center}
\includegraphics{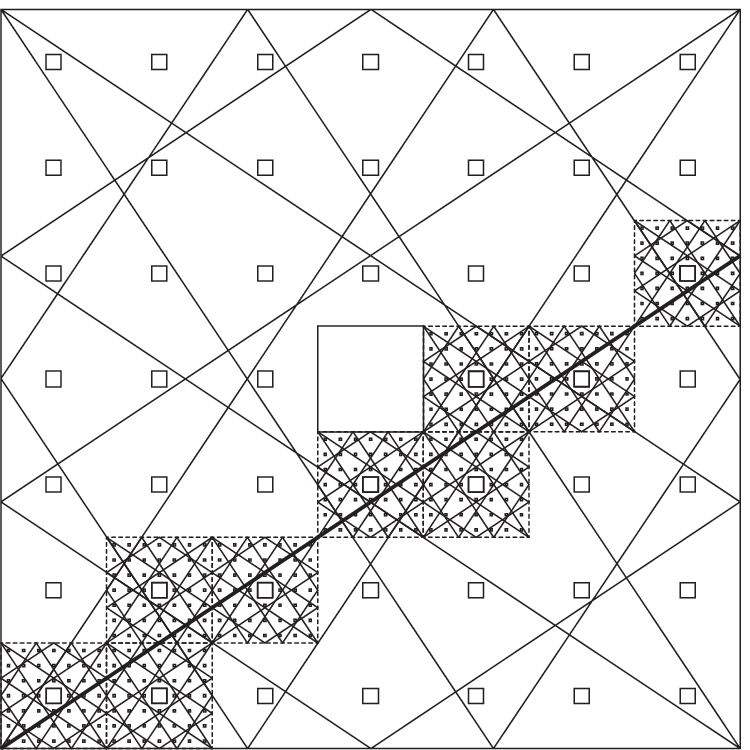}
\end{center}
\caption{In this figure, we see the scaled orbit unfolded.  Rather, we see $\frac{1}{7}\mathscr{O}^s((0,0),2/3)$ unfolded, for sufficiently large $s$.  Indeed, we see that we can recapture the original orbit $\mathscr{O}((0,0),2/3)$ by reflecting-unfolding the scaled orbit.  Moreover, reflecting-unfolding is an isometry, so no peripheral squares are intersected.}
\label{fig:unfoldingTheScaledOrbit}
\end{figure}

Suppose there exists $N\in \N$, $N\geq 3$ such that for every $n\leq N$, the orbit $\mathscr{O}((0,0),\alpha)$ avoids corners of peripheral squares with side-length $\frac{1}{7^n}$.  Then, scaling $\mathscr{O}((0,0),\alpha)$ together with $\Omega(S_{\mbf{7},N})$ by $\frac{1}{7}$, we see that the orbit of the cell $\celli{1}{k}$ avoids the peripheral square of side-length $\frac{1}{7^{N}}$ since $\mathscr{O}((0,0),\alpha)$  avoids peripheral squares of side length $\frac{1}{7^n}$,  for all $n\leq N$.  In addition, $\frac{1}{7}\mathscr{O}((0,0),\alpha)$ avoids peripheral squares of side-length $\frac{1}{7^{N+1}}$ contained in the cell, since $\mathscr{O}((0,0),\alpha)$  avoided peripheral squares of side-length $\frac{1}{7^N}$ (recall that peripheral squares of side-length $\frac{1}{7^N}$ are scaled by $\frac{1}{7}$ when constructing $S_{\mbf{7},N+1}$ from $S_{\mbf{7},N}$ via an iterated function system comprised of contraction mappings).

We may then reflect-unfold the orbit $\frac{1}{7}\mathscr{O}((0,0),\alpha)$ of the cell in $\Omega(S_{\mbf{7},N+1})$. The reflected-unfolded orbit $\frac{1}{7}\mathscr{O}((0,0),\alpha)$ must then retrace the exact path of $\mathscr{O}((0,0),\alpha)$ in $\Omega(S_{\mbf{7},N})$, but now in $\Omega(S_{\mbf{7},N+1})$.  Since reflection is an isometry, the reflected-unfolding of the orbit of the cell must then avoid peripheral squares of $\Omega(S_{\mbf{7},N+1})$.

Therefore, we have concluded that $\mathscr{O}((0,0),\alpha)$  must avoid all peripheral squares of $\Omega(S_{\mbf{7},n})$, for every $n\geq 1$.

In verifying that the example in Item (1) of \S\ref{sec:ARefinementOfTheorem4.1OfDuCaTy} constitutes a counter example to one aspect of the latter half of Theorem 4.1 of \cite{Du-CaTy}, we see how we can quickly build upon it to verify that Item (2) of \S\ref{sec:ARefinementOfTheorem4.1OfDuCaTy} is also a counter example to another aspect of the latter half of the main result in \cite{Du-CaTy}.  If $\alpha=\frac{2}{3}\in A_\mbf{5}\subseteq A_\mbf{7}$, then a nontrivial line segment of $S_\mbf{7}$ beginning at $(\frac{1}{2},0)$ with slope $\alpha$ necessarily avoids all peripheral squares of $S_\mbf{7}$.  A similar reasoning as above can be used to demonstrate the validity of this statement.  In fact, the path traversed by the orbit $\osixang{n}{(0,0)}{\frac{2}{3}}$ is identical to the path traversed by the orbit $\osixang{n}{(\frac{1}{2},0)}{\frac{2}{3}}$, for every $n\geq 0$.

The third example referred to in Item (3) of \S\ref{sec:ARefinementOfTheorem4.1OfDuCaTy} consists of a nontrivial line segment of $S_\mbf{5}$ beginning at $(\frac{1}{2},0)$ with a slope $\alpha =\frac{1}{3}\in B_\mbf{5}$.  As illustrated in Figure \ref{fig:S5TrivialLineSegment}, when a nontrivial line segment beginning at $(\frac{1}{2},0)$ with slope $\alpha=\frac{1}{3}$ is translated to $(0,0)$ and extended so as to intersect the right side of the portion of the boundary corresponding to the unit square $S_0=Q$, we see that the new segment intersects omitted squares.  Since $S_\mbf{7}$ is self-similar, such a segment will intersect infinitely many omitted squares. Hence, the new segment must be a trivial line segment of $S_\mbf{5}$.

\begin{figure}
\begin{center}
\includegraphics[scale=.95]{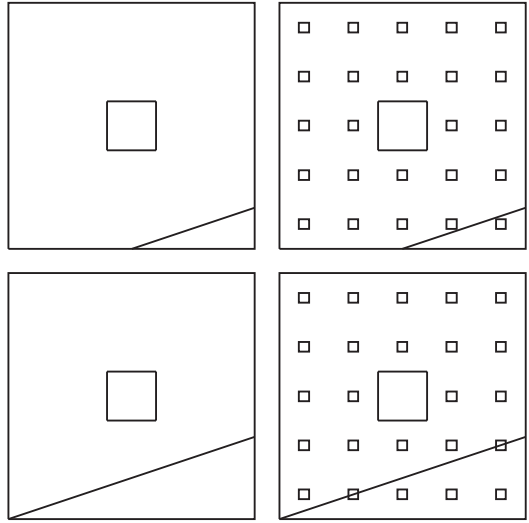}
\end{center}
\caption{An example of a nontrivial line segment of $S_\mbf{5}$ with slope $\alpha=\frac{1}{3} \in B_\mbf{5}$ starting from $\left(\frac{1}{2},0\right)$, which becomes trivial when translated to the base point $(0,0)$. }
\label{fig:S5TrivialLineSegment}
\end{figure}

Our fourth example referred to in Item (4) of \S\ref{sec:ARefinementOfTheorem4.1OfDuCaTy} consists of a nontrivial line segment of $S_\mbf{7}$ beginning at $(0,0)$ with a slope $\alpha=\frac{3}{4}$.  Consider the translation of the nontrivial line segment with slope $\alpha = \frac{3}{4}$ from $(0,0)$ to $(\frac{1}{2},0)$.  As illustrated in Figure \ref{fig:S7TrivialLineSegment}, such a segment intersects omitted squares of $S_\mbf{7}$.  Since $S_\mbf{7}$ is self-similar, it must be the case that the translated segment intersects infinitely many omitted squares of $S_\mbf{7}$. Hence, a segment starting from $(\frac{1}{2},0)$ with a slope $\alpha = \frac{3}{4}$ must be a trivial line segment.

\begin{figure}
\begin{center}
\includegraphics{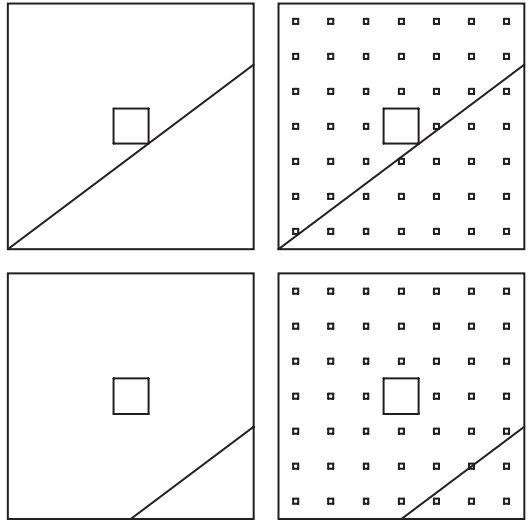}
\end{center}
\caption{An example of a nontrivial line segment of $S_\mbf{7}$ with slope $\alpha = \frac{3}{4} \in A_\mbf{7}$ starting from $(0,0)$, which becomes trivial when translated to the base point $\left(\frac{1}{2},0\right)$.}
\label{fig:S7TrivialLineSegment}
\end{figure}


\begin{definition}[Constant sequence of compatible orbits]
Given a nonnegative integer $N$, we say that a sequence of compatible orbits $\compseqsi{N}$ is a \textit{constant sequence of compatible orbits} if the path traversed by $\ofraci{n+1}$ is identical to the path traversed by $\ofraci{n}$, for every $n\geq N$.  Furthermore, we say that a sequence of compatible orbits $\compseqsi{0}$ is \textit{eventually constant} if there exists a nonnegative integer $N$ such that $\compseqsi{N}$ is constant, in the above sense.
\end{definition}

\begin{theorem}
\label{thm:sequenceOfCompatiblePeriodicOrbits}
Let $\alpha \in \slopesa{\mbfa}$.
\begin{enumerate}
\item{If $\alpha\in A_{\mbf{a-2}}$, then the sequence of compatible orbits $\compseqsixang{0}{(0,0)}{\alpha}$ is a sequence of compatible periodic orbits.}
\item{If $\alpha\in B_\mbfa$ and $r<2a^n$ is an odd positive integer, then the sequence of compatible orbits $\compseqsixang{0}{(\frac{r}{2a^n},0)}{\alpha}$ is a sequence of compatible periodic orbits.}
\end{enumerate}

Furthermore, in each case, the sequence of compatible periodic orbits is eventually constant, and its trivial limit constitutes a periodic orbit of a self-similar Sierpinski carpet billiard table $\Omega(S_\mbfa)$.
\end{theorem}

\begin{proof}
\noindent\textbf{Part 1}:  Suppose $\alpha\in A_{\mbf{a-2}}$ and an orbit of $\Omega(S_0)$ has an initial condition $((0,0),\alpha)$. Consider an orbit of $\osixang{1}{(0,0)}{\alpha}$ of $\omegasai{\mbfa}{1}$.  We claim that such an orbit avoids the peripheral square of $\Omega(S_{\mbfa,1})$.   If not, then unfolding the orbit in the tiling of the plane by $\Omega(S_{\mbfa,1})$ must intersect a peripheral square in the tiling.  Then, appropriately scaling the unfolded orbit and the tiling by a factor of $\frac{1}{a^n}$ results in a nontrivial line segment intersecting a peripheral square.  In Theorem \ref{thm:refinement}, we showed that a nontrivial line segment with a slope $\alpha\in A_{\mbf{a-2}}$ cannot intersect any point of a peripheral square of any approximation $S_{\mbfa,m}$, $m\geq 1$.  Now suppose there exists $N>0$ such that an orbit $\osixang{n}{(0,0)}{\alpha}$ avoids corners of peripheral squares of $\omegasai{\mbfa}{n}$, for all $n\leq N$.  Then, scaling $\omegasai{\mbfa}{N}$ and the orbit by $\frac{1}{a}$ results in an orbit of a cell with side length $\frac{1}{a}$.  By construction, the orbit of the cell avoids all peripheral squares contained in the cell, meaning the orbit avoids peripheral squares with side-length $\frac{1}{a^k}$, for all $k\leq N+1$. Then, reflecting-unfolding the orbit of the cell, as well as the peripheral squares of the cell, we see that the reflected-unfolded orbit $\mathscr{O}^s_{N+1}((0,0),\alpha)$ of $\omegasai{\mbfa}{N+1}$ avoids all peripheral squares. Hence, for all $n\geq 1$, the orbit $\osixang{n}{(0,0)}{\alpha}$ avoids peripheral squares of $\omegasai{\mbfa}{n}$.

\vspace{2 mm}

\noindent\textbf{Part 2}: Let $\alpha\in B_\mbfa$ and $(\xio{0},\yio{0})=(\frac{1}{2},0)$.  Then, the orbit $\osixang{1}{(\frac{1}{2},0)}{\alpha}$ avoids the  peripheral square of $\omegasai{\mbfa}{1}$. This follows from the fact that the unfolded orbit must avoid peripheral squares of the tiling of the plane by $\omegasai{\mbfa}{1}$.  Moreover, for every $m\geq 1$, $\osixang{n}{(\frac{1}{2},0)}{\alpha}$ must avoid peripheral squares of $\omegasai{\mbfa}{m}$ and $\osixang{m}{(\frac{1}{2},0)}{\alpha}$ is identical to $\osixang{0}{(\frac{1}{2},0)}{\alpha}$.  Therefore, the compatible sequence of periodic orbits $\compseqsixang{0}{(\frac{1}{2},0)}{\alpha}$ is an eventually constant sequence of compatible periodic orbits.

Let $n\geq 0$, $r$ a positive, odd integer and $r\leq 2a^n$.  Consider again the orbit $\osixang{0}{(\frac{1}{2},0)}{\alpha}$.  We know that $\osixang{0}{(\frac{1}{2},0)}{\alpha}$ avoids peripheral squares of $\omegasai{\mbfa}{n}$ for every $n\geq 1$.  Therefore, scaling $\osixang{0}{(\frac{1}{2},0)}{\alpha}$ relative to $(\frac{1}{2},0)$ by $\frac{1}{a^n}$ and making a sufficient translation so that $\xoo=\frac{r}{2a^n}$ results in an orbit of a cell with side length $\frac{1}{a^n}$ that avoids all peripheral squares contained in the cell.  By Theorem \ref{thm:refinement}, we know that the unfolding of the orbit of the cell must not intersect any peripheral squares of a tiling of the plane by the cell (as well as the scaled copy of $\omegasa{\mbfa}$ contained in the cell). Hence, the reflected-unfolding of the orbit of the cell must then also avoid peripheral squares with side-length $\frac{1}{a^k}$, for all $k\geq n$.  Moreover, by Lemma \ref{lem:BAvoids}, the reflected-unfolding of the orbit of the cell cannot intersect any corners of any peripheral squares of $\omegasa{\mbfa}$ with side-length $\frac{1}{a^j}$, $j<n$.  Therefore, $\compseqsixang{0}{(\frac{r}{2a^n},0)}{\alpha}$ is an eventually constant sequence of compatible periodic orbits.
\end{proof}

\begin{figure}
\centering
\subfigure[]{
\includegraphics[width=0.4\textwidth]{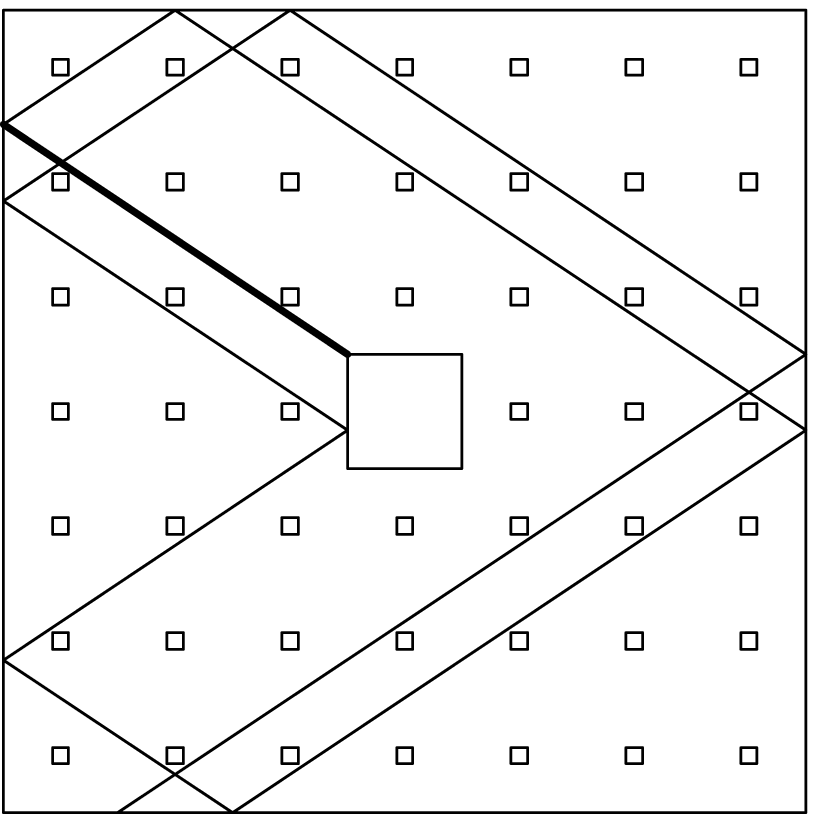}
\label{fig:S7_Slope_2_3}
}
\subfigure[]{
\includegraphics[width=0.4\textwidth]{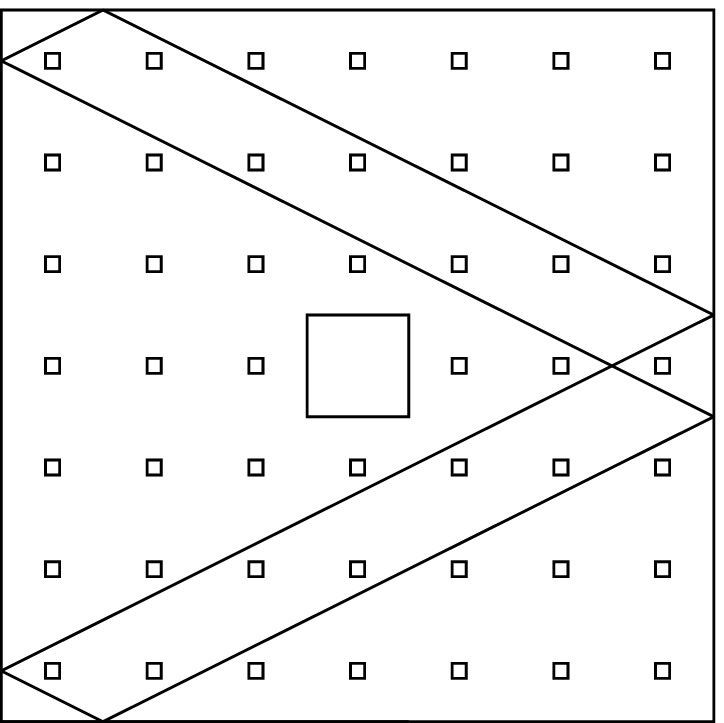}
\label{fig:S7_ Slope_1_2}
}
\caption{In $\Omega(S_{\mbf{7},2})$, $\osixang{2}{(\frac{1}{7},0)}{\frac{2}{3}}$ is a singular orbit (the line segment connecting with the singularity is shown as a segment with a greater weight than the other segments in the orbit), while $\osixang{2}{(\frac{1}{7},0)}{\frac{1}{2}}$ is a periodic orbit.}
\label{fig:singper}
\end{figure}

For an orbit with an initial direction $\alpha\in A_{\mbf{a-2}}$ starting from a corner of some cell $C_{m,a^m}$, but not from $(0,0)$, we have the following result.

\begin{proposition}
\label{prop:compatibleclosedorbits}
Consider a self-similar Sierpinski carpet billiard table$\omegasa{\mbfa}$.  Let $\alpha \in A_{\mbf{a-2}}$, $m\in \mathbb{N}$, and $k \in \{1,2,\cdots, a^m-1\}$. Then there exists an $N \in \mathbb{N}$  such that $\compseqsixang{N}{(\frac{k}{a^m},0)}{\alpha}$ is a constant sequence of compatible closed orbits.
\end{proposition}

\begin{proof}
By Theorem \ref{thm:refinement}, the orbit unfolded onto the tiling of the plane by $a^{-m}\Omega(S_{\mbfa,1})$ (that is, $\Omega(S_{\mbfa,1})$ scaled by $a^{-m}$) must avoid all peripheral squares of the tiling. Upon scaling and reflected-unfolding as in the proof of Theorem \ref{thm:sequenceOfCompatiblePeriodicOrbits}, Part 1, we deduce that for every $n\in\mathbb{N}$, $\osixang{n}{(\frac{k}{a^m},0)}{\alpha}$ avoids all peripheral squares of $\Omega(S_{\mbfa,n})$ whose side-length is less than $a^{-m}$. The claim then follows.
\end{proof}

\begin{remark}
Note that Proposition \ref{prop:compatibleclosedorbits} does not say whether $\osixang{n}{(\frac{k}{a^m},0)}{\alpha_n^0}$ is a periodic or singular orbit of $\omegasai{\mbfa}{n}$. The focus of Proposition \ref{prop:compatibleclosedorbits} is on determining a sequence of compatible \textit{closed} orbits, which can include both periodic \textit{and} singular orbits of prefractal billiards. If it hits a corner of some larger peripheral square (which has conic angle $6\pi$), then we must terminate the billiard trajectory at the corner and declare the orbit $\osixang{n}{(\frac{k}{a^m},0)}{\alpha_n^0}$ to be singular. On the other hand, if the orbit avoids corners of peripheral squares, then it will constitute a periodic orbit. See Figure \ref{fig:singper} for examples of both types of orbits.
\end{remark}





\subsection{The translation surface $\ssai{n}$}
\label{subsec:TheTranslationSurfaceSSai}

Recall from \S\ref{subsubsec:TranslationSurfacesViaRationalPolygonalBilliardTables} that a rational billiard $\Omega(D)$ can be used to construct a translation surface by appropriately identifying $2N$ many copies of $\Omega(D)$, where $N=\text{lcm}\{q_1,q_2,...,q_k\}$.  Suppose $\omegasa{\mbfa}$ is a self-similar Sierpinski carpet billiard table and $\omegasai{\mbfa}{n}$ is an approximation of $\omegasa{\mbfa}$.  Then, each interior angle has a denominator $2$.  Therefore, $N=\text{lcm}\{2,...,2\}=2$, meaning that, when four copies of $\omegasai{\mbfa}{n}$ are appropriately identified, we recover the corresponding translation surface $\ssai{n}$; see Figure \ref{fig:squareTorusWithObstaclesAndGeodesic2} for the case of a translation surface constructed from $\omegasai{\mbf{7}}{2}$, as well as an illustration of how a billiard ball would traverse the corresponding translation surface.

\begin{figure}
\begin{center}
\includegraphics[scale=.4]{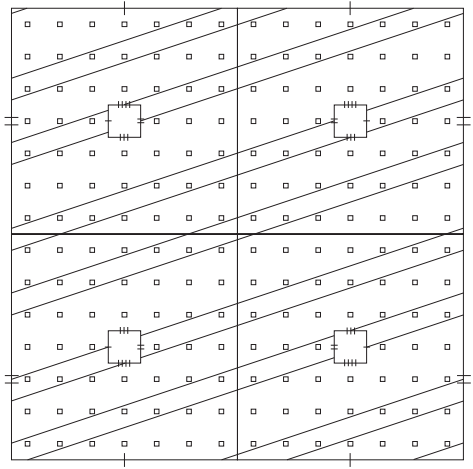}
\end{center}
\caption{We see here four copies of the prefractal approximation $\Omega(S_{\mbf{7},2})$ properly identified in such a way that the resulting translation surface is no longer a torus, but a higher genus surface.  The geodesic shown here further indicates exactly how sides of $\Omega(S_{\mbf{7},2})$ are identified in the construction of $\mathcal{S}(S_{\mbf{7},2})$. Furthermore, the way in which sides of $\Omega(S_{\mbf{7},2})$ are identified makes the geodesic equivalent to a billiard orbit under the action of the group of symmetries acting on the billiard table to produce the translation surface.  In the terminology introduced in \S \ref{subsec:mathematicalBilliards}, one could consider the geodesic in this figure to be the billiard orbit unfolded in the translation surface.}
\label{fig:squareTorusWithObstaclesAndGeodesic2}
\end{figure}

As previously shown in Figure \ref{fig:ReflectionDefinedInCorners}, reflection in a corner with an angle measuring $\frac{\pi}{2}$ in the square can be made well-defined. The justification for this is provided by the fact that in the associated translation surface, a vertex of the unit square is identified with three other vertices in such a way that the conic angle about the conic singularity is $2\pi$.  In Figure \ref{fig:squareTorusWithObstaclesAndGeodesic}, we see the effect that a removable singularity (i.e., a singularity with conic angle $2\pi$) has on the straight-line flow.  Specifically, each point of $\mathcal{S}(S_{\mbfa,n})$ corresponding to a corner of the zeroth level approximation $\Omega(S_0)$ (or, the unit square billiard table), is a removable singularity of the translation surface; each point of $\mathcal{S}(S_{\mbfa,n})$ corresponding to a corner of a peripheral square of $\Omega(S_{\mbfa, n})$ is a nonremovable singularity of the translation surface, with the conic angle about each nonremovable singularity measuring $6\pi$.

\begin{figure}
\begin{center}
\includegraphics[scale=.4]{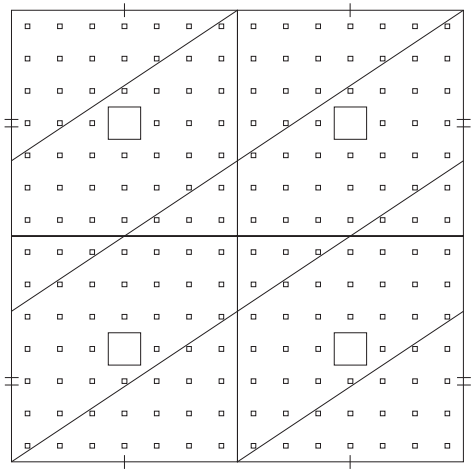}
\end{center}
\caption{In this figure, we see a geodesic that of the translation surface $\mathcal{S}(S_{\mbf{7},2})$ that intersects removable conic singularities of the surface.  As such, the straight-line flow can continue unimpeded. Since the geodesic is equivalent to a billiard orbit in $\Omega(S_{\mbf{u},2})$, reflection in the corners with angles measuring $\frac{\pi}{2}$ can be determined from how an equivalent geodesic passes through the corresponding removable singularity.}
\label{fig:squareTorusWithObstaclesAndGeodesic}
\end{figure}

Using the formula given in Equation (\ref{eqn:genus}), the genus $g_n$ of a translation surface $\mathcal{S}(S_{\mbfa,n})$ is calculated as follows.
\begin{align}
g_n &= 1+2\sum_{m=1}^n a^{m-1}.
\end{align}
\noindent Then, for a sequence of rational (prefractal) billiard tables $\{\Omega(S_{\mbfa,n})\}_{n=0}^\infty$ approximating $\Omega(S_\mbfa)$, we have that $\lim_{n\to\infty} g_n = \infty$.  Hence, if there is a translation surface associated with the (yet to be properly defined) self-similar billiard table $\omegasa{\mbfa}$, then such a translation surface presumably has infinite genus.  On the other hand, if such a surface is given as a suitable limit of translation surfaces $\ssai{n}$, then such a surface will presumably have no surface area. In other words, it is reasonable to expect that a suitable translation surface cannot be constructed as a suitable limit of a sequence of translation surfaces $\ssai{n}$.

\section{Concluding remarks and open questions}
\label{sec:ConcludingRemarks}

In this paper we have made progress in identifying a family of periodic orbits in the self-similar carpet billiard $\omegasa{\mbfa}$. As of now, an orbit of a self-similar Sierpinski carpet billiard is the  trivial limit of an eventually constant sequence of compatible periodic orbits.  As an interim goal, we would like to identify sequences of compatible periodic orbits that are not eventually constant, but converge (in a suitable sense) to a set that constitutes a periodic orbit of a self-similar Sierpinski carpet billiard table $\omegasa{\mbfa}$.   Ultimately, we would like to construct a well-defined (topological) dynamical system on a self-similar Sierpinski carpet billiard table and determine whether or not there exists a topological dichotomy for the billiard flow.

\begin{question}
If one can construct a well-defined notion of reflection for $\omegasa{\mbfa}$, is it possible to show that, for a fixed direction $\alpha^0$, the billiard flow in the direction $\alpha^0$ is either closed or dense in $\omegasa{\mbfa}$?
\end{question}

In constructing the translation surface corresponding to $\Omega(S_{\mbfa,n})$, we have come up with a concrete example of a sequence of translation surfaces whose genera increase to infinity as $n\to\infty$. However, because $S_\mbfa$ has zero Lebesgue area, it is unclear exactly how to describe the limit of the sequence of translation surfaces as a translation surface, let alone a surface with any specific structure.

\begin{question}
\label{ques:suitableNotionOfLimitTranslationSurface}
Can one apply a suitable notion of limit to a sequence of translation surfaces $\{\ssai{i}\}_{i=0}^\infty$ so as to recover a `surface' with some structure that is analogous to that of a translation surface?
\end{question}

If Question \ref{ques:suitableNotionOfLimitTranslationSurface} can be answered in the affirmative, a natural question to ask is whether one can determine a flow on $\mathcal{S}(S_\mbfa)$, and whether such a flow can be shown to be equivalent to the (yet to be determined) billiard flow on $\omegasa{\mbfa}$.

A related (and interesting) problem involves investigating the behavior of a sequence of compatible orbits of billiard tables which share a similar construction with the self-similar Sierpinski carpet billiard table $\Omega(S_\mbfa)$, but occupy nonzero area. A clear candidate is the non-self-similar Sierpinski carpet $S_\mbfa$ with $\mbfa = (a_1^{-1}, a_2^{-1},\cdots)$ being a sequence in $l^2$ (\emph{i.e.,} $\sum_{j=1}^\infty a_j^{-2} <\infty$). Since $\limsup \mbfa =0$ in this setting, \cite[Theorem 5.3(a)]{Du-CaTy} implies that $S_\mbfa$ contains nontrivial line segments of every rational slope, and contains no nontrivial line segments of any irrational slope. One may then begin to discuss and analyze the behavior of a sequence of compatible orbits with each orbit having an initial direction given by $\alpha^0\in \slopesa{\mbfa}=\mathbb{Q}$.

A related Sierpinski carpet billiard we may consider investigating is a so-called \textit{fat Sierpinski carpet} billiard table, constructed as follows. Begin by fixing a positive odd integer $a\geq 3$, and partition the unit square into $a^2$ many squares of side-length $a^{-1}$. Rather than removing the open middle square of side-length $a^{-1}$ as in the first-stage of the construction of a self-similar Sierpinski carpet, one removes an open square of side-length $(1-\delta) a^{-1}$, $0<\delta<1$. In the second stage of the construction, a square of side-length $(1-\delta)a^{-2}$ centered within each of the remaining $a^2-1$ squares of side-length $a^{-1}$ is removed. This process is then repeated ad infinitum. In effect, each removed square of side-length $(1-\delta) a^{-n}$ is surrounded by a solid square annulus of width $(\delta/2) a^{-n}$. Such a carpet admits nontrivial line segments of irrational slopes, and there will be additional base points from which a nontrivial line segment can emanate. We can then ask the following question.

\begin{question}
How does the existence and behavior of a sequence of compatible dense orbits depend on $\delta$? Moreover, can one recover a dense orbit of the self-similar Sierpinski carpet billiard $\Omega(S_\mbfa)$ in the limit as $\delta\to 0$?
\end{question}


The development of our understanding of the Sierpinski carpet billiard table may be possibly aided by studying the results for other types of billiard tables.  In particular, the so-called wind-tree billiard tables (and generalizations of them) investigated in \cite{CoGut,De,HuLeTr} are very reminiscent of a Sierpinski carpet billiard table.

Alternatively, we may be able to approach the problem of determining billiard dynamics on $\Omega(S_\mbfa)$ by taking an algebraic perspective. For each translation surface $\mathcal{S}(S_{\mbfa,n})$, there is an associated\textit{Veech group}. A Veech group of a translation surface $\mathcal{S}(D)$ determined from a rational billiard table $\Omega(D)$ is the stabilizer of $\mathcal{S}(D)$.  The work of \cite{We-Sc} may prove useful in determining a well-defined translation structure on a suitable limit of a sequence of translation surfaces $\mathcal{S}(S_{\mbfa,n})$.  In particular, analyzing the inverse limit of an inverse limit sequence of Veech groups $\{\Gamma(S_{\mbfa,n})\}_{n=0}^\infty$ (assuming such an inverse limit sequence can be properly constructed) may yield a group that is the stabilizer of some surface.  With such a surface and an associated stabilizer defined, we may begin to investigate whether or not the proposed surface contains saddle connections. The directions of such saddle connections may then indicate directions for which closed geodesics are occurring.

\section*{Acknowledgements}

We wish to thank Jeremy Tyson for several useful discussions regarding his work with Estibalitz Durand-Cartagena. We also thank our respective Ph.D. advisors, Robert S. Strichartz and Michel L. Lapidus, for their support of the project.

\end{document}